\documentclass[11pt]{article}
\usepackage{amssymb,amsmath,amsfonts}
\usepackage{graphicx,color,enumitem}
\usepackage{subfig}
\usepackage{amsthm} 
\usepackage{bm}
\usepackage[round]{natbib}
\usepackage{soul}
\usepackage{geometry}
\usepackage{stmaryrd}
\usepackage{pifont}   
\usepackage{tikz}
\usetikzlibrary{decorations.pathreplacing}
\usetikzlibrary{patterns}

\usepackage{mathtools}
\usepackage[ruled,vlined]{algorithm2e}
 \mathtoolsset{showonlyrefs}
\usepackage[colorinlistoftodos, textwidth=4cm, shadow]{todonotes}
\RequirePackage[colorlinks=true,citecolor=blue,urlcolor=blue]{hyperref}

\usepackage{bbm}

\newcommand{\E}{{\mathbb E}}
\newcommand{\F}{{\mathbb F}}

\renewcommand{\P}{{\mathbb P}}

\newcommand{\R}{{\mathbb R}}

\newcommand{\Acal}{{\mathcal A}}
\newcommand{\Bcal}{{\mathcal B}}

\newcommand{\Dcal}{{\mathcal D}}

\newcommand{\Fcal}{{\mathcal F}}

\newcommand{\Lcal}{{\mathcal L}}

\newcommand{\Mcal}{{\mathcal M}}
\newcommand{\Ncal}{{\mathcal N}}

\newcommand{\Scal}{{\mathcal S}}
\newcommand{\Tcal}{{\mathcal T}}

\newcommand{\vertiii}[1]{{\left\vert\kern-0.25ex\left\vert\kern-0.25ex\left\vert #1 
    \right\vert\kern-0.25ex\right\vert\kern-0.25ex\right\vert}}
\newcommand{\V}{\mathrm{Var}}

\newcommand{\red}{\color{red}}
\newcommand{\blue}{\color{blue}}
\definecolor{darkgreen}{rgb}{0,0.7,0}

\newcommand{\iii}{{\vert\kern-0.25ex\vert\kern-0.25ex\vert}}
\newcommand{\bec}[1]{\begin{equation} \begin{cases} #1\end{cases} \end{equation}}
\newcommand{\bes}[1]{\begin{equation} \begin{split} #1\end{split} \end{equation}}
\renewcommand{\c}{\alpha}

\definecolor{blue0}{RGB}{0,77,153} 
\definecolor{red0}{RGB}{179,0,77} 
\definecolor{green0}{RGB}{134,219,76} 
\definecolor{gray0}{RGB}{84,97,110}

\newtheorem{theorem}{Theorem}

\newtheorem{definition}[theorem]{Definition}

\newtheorem{lemma}[theorem]{Lemma}

\newtheorem{proposition}[theorem]{Proposition}
\newtheorem{remark}[theorem]{Remark}
\theoremstyle{definition}
\numberwithin{equation}{section}
\numberwithin{theorem}{section}

\newcommand{\td}{{t-d}}
\newcommand{\intd}{\int_{t-d}^t}
\newcommand{\intdd}{\int_{[t-d,t]^2}}

\begin{document}

\title{Linear-quadratic stochastic delayed control and deep learning resolution \footnote{We would like to thank Salvatore Federico and Huyên Pham for their useful comments and remarks that helped to improve this article.}}
\author{William Lefebvre\footnote{BNP Paribas Global Markets, Université de Paris and Sorbonne Université, Laboratoire de Probabilités, Statistique et Modélisation (LPSM, UMR CNRS 8001), Building Sophie Germain, Avenue de France, 75013 Paris, \sf wlefebvre at lpsm.paris}  \and Enzo Miller\footnote{Universit\'e de Paris and Sorbonne Universit\'e, Laboratoire de Probabilit\'es, Statistique et Mod\'elisation (LPSM, UMR CNRS 8001), 
Building Sophie Germain, Avenue de France, 75013 Paris,  \sf  enzo.miller at polytechnique.org} }
\maketitle

\begin{abstract}
    We consider a class of stochastic control problems with a delayed control, both in drift and diffusion, of the type $dX_t = \alpha_{t-d}(bdt + \sigma dW_t)$. 
    We provide a new characterization of the solution in terms of a set of Riccati partial differential equations. Existence and uniqueness are obtained under a sufficient condition expressed directly as a relation between the horizon $T$ and the quantity $d(b/\sigma)^2$. Furthermore, a deep learning scheme\footnote{The code is available in a \href{https://colab.research.google.com/drive/1hmh3VJNZS0lErwXkhnGNWBPNJF_5tbYC?usp=sharing}{IPython notebook.}} is designed and used to illustrate the effect of delay on the Markowitz portfolio allocation problem with execution delay.
\end{abstract}

\vspace{5mm}


\vspace{5mm}


{
  \hypersetup{linkcolor=black}
  \tableofcontents
}

\section{Introduction}
The control of systems whose dynamic contains delays on the state and/or control has attracted the attention of the optimization and engineering communities in the last decades due to its wide variety of applications, allowing to tackle problems where the past of a system influences its present or where an agent controls a system with a latency.
As a non-exhaustive list of applications we may cite the following papers, classified by their applications domain: Engineering (\cite{tian1999control}, \cite{huzmezan2002time}); {Advertising} (\cite{sethi1974sufficient}, \cite{pauwels1977optimal},  \cite{gozzi200513}, \cite{gozzi2009controlled}); {Learning by doing with memory effect} (\cite{d2012business}); {Growth model with lags between investment decision and project completion} (\cite{asea1999time}, \cite{hall1977investment}, \cite{jarlebring2007lambert}, \cite{bambi2008endogenous}, \cite{bambi2012optimal}); Investment (\cite{tsoukalas2011time}, \cite{kydland1982time}). More recently, the introduction of delayed control together with mean-field effects was studied (\cite{carmona2018systemic}, \cite{fouque2019deep}) and new machine learning methods have been designed to numerically solve stochastic control problems with delay (\cite{han2021recurrent}). We also refer to the monograph \cite{sipahi2011stability} to find literature on the various effects of delays on traffic flow modelling, chemical processes, population dynamics, supply chain, etc.

In the optimal control community, two main approaches have emerged: the \textit{structural state method} and the \textit{extended state method}, and we refer to \citet[Part II, Chapter 3]{bensoussan2007representation} for the study of the latter in the deterministic case and \cite{fabbri2014infinite} for the structural state approach in the stochastic case. Let us also mention the paper by \cite{fabbri2014infinite} for an overview and exhaustive list of  references.

In this paper, we aim at  studying the challenging case where there is a delayed control both in the drift and volatility. Except in \cite{fabbri2014infinite}, this situation is not treated theoretically nor numerically in  the references above. The main difficulty comes from the fact that the natural formulation of a control problem with delayed control involves a \textit{boundary control problem}. Indeed, assume for instance that $X$ denotes a state variable following the simple dynamic  $\dot{X}_t = \c_{t-d}$, where $\c$ denotes the control. For any time $t$ and index $s \in [-d,0]$, set $u_t(s) = \c_{t+s}$, the \textit{memory} of the control $\c$. Then, note that $\partial_t u_t(s) = \partial_s u_t(s)$ and $u_t(0)=\c_t$. Thus, the natural infinite dimensional formulation of the controlled system is 
\begin{align}
    \text{State eq. on $(X,u)$} &\left\{
        \begin{array}{ll}
            \dot{X}_t = \Mcal u_t  &  \\
            (\partial_t  - \partial_s)  u_t(s) = 0, &
        \end{array}
    \right. 
\\
 \text{Boundary constraint} &\left\{
    \begin{array}{ll}
        \Bcal u_t = \c_t,  &  
    \end{array}
\right. \\
 \text{Initial conditions} &\left\{
    \begin{array}{ll}
        u_0(s) = \gamma_s, \quad X_0 = x,  &  
    \end{array}
\right. \\
\end{align}
where $t,s \in [0,T]\times[-d,0]$, $\Mcal u := u(-d)$, $\Bcal u := u(0)$ and $\gamma$ is the initial value of the control over $[-d,0]$. Consequently, any delayed controlled problem where the delay appears in the control variable can be recast as a boundary control problem whose geometry is parametrized by the delay $d$, see Figure \ref{fig:domain}.\\ 

\noindent\textbf{Main contributions.} Our goal is to shed some lights on the difficulty related to delayed control on the volatility and to provide a practical and simple tool for designing a numerical scheme practitioners can play with.  In this paper, we study the most simple linear-quadratic control problem with delayed control both in drift and volatility. The optimal feedback control and the value function are given in terms of Riccati partial differential equations and the extended state $(x,u) \in H=\R \times L^2([0,T], \R)$, where $x$ denotes the position and $s \in [-d,0] \mapsto u(s)$ the memory of the control. The existence and uniqueness of these latter are proven under a condition, emerging from the delay feature, involving the drift $b$, the volatility $\sigma$, the delay $d$ and the horizon $T$. Finally, we adopt a deep learning approach in the spirit of the papers by \cite{raissi2019physics} (Physics Informed Neural Network) and \cite{sirignano2018dgm} (Deep Galerkin) to  propose a numerical scheme. Our results are illustrated on the celebrated Markowitz portfolio allocation problem where we take into account execution delay. We believe the semi-explicit resolution of infinite dimensional control problem by means of deep learning method will open the door to several interesting applications such as quick simulations of richer models, precise benchmarking of \textit{reinforcement learning} algorithms,  etc.\\

\noindent\textbf{Outline of the paper.} The rest of the paper is organized as follows: In Section \ref{section:heuri} we formulate the stochastic delayed control problem and derive an heuristic approach through a lifting in an infinite dimensional space, namely the \textit{extended state space} in the spirit of \cite{ichikawa1982quadratic}, but without the use of semi-group theory. We state in Section \ref{section:verification_simplified} a verification theorem and prove existence and uniqueness results for the Riccati PDEs.
A deep learning based numerical scheme with two applications on Markowitz portfolio allocation is given in  Section \ref{section:nn}, with a detailed analysis of the effect of the delay feature on the allocation strategy. \\

\noindent\textbf{Notations.}~~\\
Given a probability space $(\Omega, \Fcal, \mathbb{P})$, a filtration $\mathbb{F} = (\Fcal_{ t})_{t\geq  0}$ satisfying the
usual conditions and $a<b$  two real numbers, we denote by
\bes{
    L^2([a,b], \R) &= \left\{ Y : [a,b] \mapsto \R, \text{ s.t. } \int_a^b |Y_t|^2 dt < \infty \right\}, \\
     L_{\mathbb{F}}^2([a,b], \R) &= \left\{ Y : \Omega \times [a,b] \mapsto \R, \mathbb{F}-\text{prog. measurable s.t. } \E\left[\int_a^b |Y_t|^2 dt \right] < \infty \right\}. \\
}
Here $|\cdot|$ denotes the Euclidean norm on $\R$ or $\R^d$, and $H=\R \times L^2([0,T], \R)$ denotes the \textit{extended state space} endowed with the scalar product $\langle x, y \rangle_H = x_0 y_0 + \int_{-d}^0 x_1 (s) y_1 (s) ds$. For any $z=(x,u) \in H$, we use the notation $z_0=x$ and $z_1=u$.

\section{Formulation of the problem and heuristic approach}
\label{section:heuri}
Let $\left( \Omega, \mathcal{F}, \F := (\mathcal{F}_t)_{t \leq 0}, \P \right)$ be a complete filtered probability  space on which a real-valued Brownian motion $(W_t)_{t\leq 0}$ is defined and consider the simple system defined on $[0,T]$ by the following dynamics 
\bec{
    \label{eq:dynamic_heuri}
    &d X_t^{\c} = \c_{t-d}\left( b dt + \sigma dW_t\right),  \qquad 0\leq t \leq T,\\
    &X_0 = x, \quad \c_s = \gamma(s), \qquad s \in [-d, 0],
}
endowed with the cost functional
\bes{
    \label{eq:cost_functional_heuri}
    J(\c) = \E\left[(X^{\c}_T)^2 \right],
}
where $\gamma \in L^2([-d,0], \R)$ and $\c$ models the control chosen in the set of admissible strategies $\Acal$:
\bes{
    \mathcal{A}= \left\{\c \in L_{\F}^{2}([0,T], \R) \text{ such that } \eqref{eq:dynamic_heuri} \text{ has a solution satisfying } \E\left[\sup_{t\leq T}|X_t^\c|^2\right]< \infty\right\}.
}
For any $0 \leq a<b \leq T$, we also define the set $\mathcal{A}_{a,b}$ as the restriction of $\Acal$ to  $L_{\F}^{2}([a,b], \R)$.

\begin{remark}
    At this point, {we may expect a priori that} the optimization problem \eqref{eq:dynamic_heuri}-\eqref{eq:cost_functional_heuri} admits an optimizer provided $\sigma \neq 0$, even if the control is not directly penalized. {The intuition behind this a priori belief is that}, the more $\c$ is aggressive in bringing $X$ to 0, the more the variance of $X$ increases due to the diffusion term. It is the case in the classical LQ stochastic optimization problem with controlled volatility such as
    \bes{
        dX_t^\c &= \c_t(b dt + \sigma dW_t), \qquad t \leq T,\\
        X_0&=x,\\
        J(\c) &= \E[(X^\c_T)^2],
    }
    where the optimal control reads $\c^*_t = -\frac{b}{\sigma^2}X_t^{\c^*}$ and the value function $V_t = e^{(t-T)\frac{b^2}{\sigma^2}}$. A surprising finding in our paper is the necessity for a more restricting {condition} on the diffusion coefficient due to the delay feature, see Proposition \ref{prop:ricatti}.
\end{remark}

\begin{remark}
    In the rest of the paper we focus on the one dimensional case with delayed control both in drift and volatility which features the main difficulties related to the presence of the delay. Although Proposition \ref{prop:ricatti} {concerning the} existence and uniqueness of a Riccati-PDE system does not directly extend to the multidimensional case, the verification Theorem can easily be adapted to the multidimensional case with delayed state and control. 
\end{remark}

The first step consists in lifting the dynamics in the infinite dimensional Hilbert space $H=\R \times L^2([0,T], \R)$, where the system is naturally Markovian. To do so, denote $u_t(s) = \alpha_{t+s}$ for any $t\leq T$ and $s\in [-d, 0]$, a transport of the control. The dynamics \eqref{eq:dynamic_heuri} then reads 
\bec{
    \label{eq:infinite}
    &d Z^\c_t = A Z^\c_t dt + BZ^\c_t d W_t + C d\c_t, \qquad  0\leq t \leq T,\\
    & Z_0 = (x, \gamma),
}
where $Z^\c$ is defined as the $H = \R \times L^2([0,T], \R)$-valued random process $Z^\c_t = (X_t^\c, u_t(\cdot))$ and 
\begin{align}
    A = \begin{pmatrix}
        0 && b\delta_{-d} \\
        0 && \partial_s 
        \end{pmatrix},
        &&
    B = \begin{pmatrix}
        0 && \sigma\delta_{-d} \\
        0 && 0
        \end{pmatrix},
        &&
    C = \begin{pmatrix}
        0 \\
        1_{0}(\cdot)
        \end{pmatrix}.
\end{align}
Let $V$ be the value function 
\bes{
    V(t, z) = V(t, (x,u)) = \inf_{\substack{\c\in \Acal_{t,T} }} \E[ (Z_T^\c)_0^2] =\inf_{\substack{\c\in \Acal_{t,T} }} \E[ (X_T^\c)^2], \qquad z \in H,
}
where $Z^\c$ denotes the solution to \eqref{eq:infinite} starting from $z=(x,u)$ at time $t$. Then, assuming $V\in C^{1,2}\left([0,T]  \times L^2\left([-d, 0]\right), \R \right)$, the dynamic programming principle reads 
\bes{
     V(t, z) &=  \inf_{\substack{\c\in \Acal_{t,t+h} }} \E[ V(t+h, Z^{\c}_{t+h})]\\
             &=  \inf_{\substack{\c \in \Acal_{t,t+h} }}  \E \Bigg[  V(t, z) + \int_t^{t+h} \partial_t V(s,Z_s^{\c}) ds + \int_t^{t+h} \partial_z V(s,Z_s^{\c}) dZ_s^{\c} \\
             & \; \; \; \; + \frac{1}{2} \int_t^{t+h} \partial_z^2 V(s,Z_s^{\c}) d \langle Z^{\c} \rangle_s  \Bigg].\\
             &= \inf_{\substack{\c\in \Acal_{t,t+h} }}  \E \Bigg[  V(t, z) + \int_t^{t+h} \partial_t V(s,Z_s^{\c}) ds + \int_t^{t+h} \partial_z V(s,Z_s^{\c}) (AZ^{\c}_s ds + C d\c_s) \\
             & \; \; \; \; + \frac{1}{2} \int_t^{t+h} \partial_z^2 V(s,Z_s^{\c}) d \langle Z^{\c} \rangle_s  \Bigg],\\
}
Note that $1_{0}(\cdot) = 0_{L^2}$. As a result, simplifying by $V(t, z)$, dividing by $h$ and letting $h \to 0$ yields (informally) the Hamilton-Jacobi equation
\bes{
    \label{eq:hjb}
    \partial_t V + \inf_{\c \in \R} \{ \partial_z V Az + \partial_z^2 V  (Bz \otimes Bz)\} &=0, \qquad t \leq  T, \quad z \in L^2([-d, 0], \R), \\
    V(T, z) &= z_0^2.
}
Recall that in equation \eqref{eq:hjb}, we have $z_1(0) = u(0)=\c$. Let us now assume that the value function $V$ is of the following form 
\bes{
    V(t, z) = \langle P_t z, z \rangle_{H},
}
where $P \in C([0,T], \Lcal(H,H))$ is a self-adjoint bounded positive operator valued function of the form 
\bes{
    P_t : (x,\gamma(\cdot))\mapsto \begin{pmatrix}
            P_{11}(t)x +  \int_{-d}^0 P_{12}(t,s)  \gamma(s)  ds \\
            P_{12}(t,\cdot) x + P_{\hat{22}}(t, \cdot) \gamma(\cdot) + \int_{-d}^0 P_{22}(t,\cdot,s) \gamma(s)ds
        \end{pmatrix}.
}
Thus, for any $z= (x,u) \in H$ such that $u(0)=\c$ and $t\leq T$, equation \eqref{eq:hjb} reduces to
\bes{
    \label{eq:hjb_bis}
    \langle \dot{P}_t z,z \rangle_H + \inf_{\c \in \R} \{ \langle P_t A z, z \rangle_H +  \langle P_t z,A z \rangle_H + \langle P_t Bz,B z \rangle_H \} = 0.
}
Furthermore, using the boundary condition $u(0)=\c$ together with integration by part, we have
\bes{
    \label{eq:paz}
    \langle P_t z,A z \rangle_H &= (P_t z)_0 (Az)_0 + \int_{-d}^0 (P_t z)_1(s) (Az)_1(s) ds \\
    & = bu(-d) \left(P_{11}(t)x + \int_{-d}^0 P_{12}(t,s)u(s) ds \right) + \c x P_{12}(t,0) - u(-d) x P_{12}(t,-d) \\  
    & \;\;\;\; - x \int_{-d}^0 \partial_s P_{12}(t,s)  u(s) ds + \c \int_{-d}^0 P_{22}(t,0,s)u(s)ds   \\
    & \;\;\;\; - u(-d) \int_{-d}^0 P_{22}(t,-d, s)u(s)ds  - \int_{-d}^0 \int_{-d}^0 \partial_s P_{22}(t,s,r) u(s)u(r) ds dr  \\
    & \;\;\;\; + \c^2 P_{\hat{22}}(t,0) -u(-d)^2 P_{\hat{22}}(t,-d) - \int_{-d}^0 \partial_s P_{\hat{22}}(t,s)ds, \\
}
and
\bes{
    \label{eq:pbz}
    \langle P_t B z,B z \rangle_H &= \sigma^2 P_{11}(t) u(-d)^2.
}
\begin{remark}
    In \eqref{eq:paz}, along with the integration by part, formulas such as $u \partial_s u = \partial_s u^2$ were (formally) used. However, as it appears in the verification Theorem \ref{T:verif}, the feedback optimal control obtained is as regular as the  controlled process $X^\c$ and thus as regular as the Brownian motion $W$. This is why our approach is only heuristic and justifies the need for the verification Theorem $\ref{T:verif}$.
\end{remark}
As a consequence, the minimizer of the Hamiltonian in \eqref{eq:hjb_bis} reads
\bes{
    \label{eq:optim_c}
    \c^*(t,z) = -\frac{1}{P_{\hat{22}}(t,0)} \left( x P_{12}(t,0) + \int_{-d}^0 P_{22}(t,0,s)u(s)ds  \right).
}
\begin{remark}
    Note that when $d \to 0$, then $\alpha^*(t,z) \to -\frac{b}{\sigma^2}x$ which agrees with the optimal strategy in the undelayed case.
\end{remark}

Combining \eqref{eq:hjb_bis}, \eqref{eq:paz} and \eqref{eq:pbz} yields the set of Riccati partial differential equations
\begin{align}
    \label{eq:a}
	     &\dot{P}_{11}(t)  = \frac{P_{12}(t,0)^2}{P_{\hat{22}}(t,0)}, &&
	     (\partial_t-\partial_s)(P_{12})(t,s) = \frac{P_{12}(t,0)P_{22}(t,s, 0)}{P_{\hat{22}}(t,0)},\\
	     &(\partial_t - \partial_s)(P_{\hat{22}})(t,s)=0, &&
	    (\partial_t - \partial_s-\partial_r)(P_{22})(t,s,r) = \frac{P_{22}(t,s,0)P_{22}(t,0,r)}{P_{\hat{22}}(t,0)},
\end{align}
accompanied by the boundary conditions 
\begin{align}
    \label{eq:b}
	      &P_{12}(t,-d) = b P_{11}(t), &&
	      P_{\hat{22}}(t,-d) = \sigma^2 P_{11}(t), \\
	      &P_{22}(t,s, -d)= b P_{12}(t,s), && P_{22}(t,-d, r)= b P_{12}(t,r),
\end{align}
and the final conditions
\begin{align}
    \label{eq:c}
	      &P_{11}(T) = 1, &&
	        P_{12}(T,s) = P_{\hat{22}}(T,s) = P_{22}(T,s,r)=0,
\end{align}
for almost every $s,r \in [-d,0]$. 
\begin{remark}
    Looking at the expression \eqref{eq:optim_c}, we can already guess some effects of the existence of a delay on the optimal strategy. Indeed, from \eqref{eq:b} one notes that $P_{12} \approx b$, $P_{\hat{22}} \approx \sigma^2$, ${P}_{22} \approx b^2$, and we may write 
    \bes{
        \c \approx \frac{-1}{\sigma^2} (bx + d b^2 \alpha) \approx \frac{-bx}{\sigma^2(1+d (b/\sigma)^2)}.
    }
    In Section \ref{section:nn}, we illustrate numerically the various effects of the delayed control through two examples of Markowitz portfolio allocation with execution delay.
\end{remark}
Note that due to the existence of the delay, the value function is independent of the control chosen after $T-d$, so that $P_{12}(t,s) = P_{\hat{22}}(t,s) =P_{22}(t,s,r)=0$ whenever $t+s \geq T-d$ or  $t+r \geq T-d$. Similarly, the optimal control defined in \eqref{eq:optim_c} is ill defined on $[T-d,T]$ so we decide to set to zero the control after time $T-d$ and rewrite 
\bes{
    \label{eq:optim_c_}
    \c^*(t,z) = -\frac{1_{t\leq T-d}}{P_{\hat{22}}(t,0)} \left( x P_{12}(t,0) + \int_{-d}^0 P_{22}(t,0,s)u(s)ds  \right).
}

Thus, to make sense of the set of Ricatti partial differential equations \eqref{eq:a}-\eqref{eq:b}-\eqref{eq:c} and the optimal control \eqref{eq:optim_c_}, we adopt the convention $0^2/0 = 0$ and define the concept of solution as follows 

\begin{definition}
\label{def:sol_E_i_simplified}
 A 4-uplets $P = (P_{11}, P_{12}, P_{\hat{22}}, P_{22})$ is said to be a solution to \eqref{eq:a}-\eqref{eq:b}-\eqref{eq:c} if $P_{11} : [0,T] \mapsto \R$, $P_{22},  P_{\hat{22}} : [0,T]\times [-d, 0] \mapsto \R$ and $P_{22} : [0,T]\times [-d, 0]^2 \mapsto \R$  are piecewise absolutely continuous functions satisfying \eqref{eq:a}-\eqref{eq:b}-\eqref{eq:c} with $P_{\hat{22}}(t) >0$ for any $t <  T-d$.
\end{definition}

The reason we chose \textit{piecewise} absolutely continuous functions as our set  of functions is because we expect the kernel $P$ to be discontinuous. To illustrate this consideration, cut the domain $\mathcal{D}$ into three pieces $\mathcal{D} = [0,T] \times [-d, 0]^2 = \Dcal_a \cup \Dcal_b \cup \Dcal_c$ as represented in Figure \ref{fig:domain}, with
 \bes{
 \mathcal{D}_{a} &= [0,T-d] \times [-d, 0]^2, \\
 \mathcal{D}_{b} &= \{(t,s,r) \in \Dcal \quad \text{s.t.} \quad t > T-d , \quad t + s \vee r < T-d \},  \\
 {\mathcal{D}_c} &= \{(t,s,r) \in \Dcal \quad \text{s.t.} \quad t > T-d , \quad t + s \vee r \geq T -d \}  \\
 }
 and note that, necessarily, $P_{12}, P_{\hat{22}}$ and $P_{22}$ are null on $\Dcal_c$ but not on the remaining domain, see also the numerical simulations in  Figure \ref{eq:kernels}.
\newcommand\pgfmathsinandcos[3]{%
  \pgfmathsetmacro#1{sin(#3)}%
  \pgfmathsetmacro#2{cos(#3)}%
}
\begin{figure}[!h]
\begin{minipage}{0.48\textwidth}
     \centering
     \begin{tikzpicture}[scale=0.62]
        \draw[thick,-stealth,black, ->] (0.,-1.0) -- (0.,10.) node[above] {$t$}; 
        \draw[thick,-stealth,black, ->] (1.,0.) -- (-5.,0) node[below] {$s$};
        \draw (0., 0.) node[below left] {$0$};
    
        \draw (-3., -0.1) node[below] {\footnotesize $-d$};
        \draw (0., 6.) node[right] {\footnotesize $T-d$};
        \filldraw[color=black, fill=black!9, thick] (0,0) rectangle (-3.,6.);
        \filldraw[color=blue, fill=blue!9, thick] (0.,6.) -- (-3.,8.) -- (-3.,6.) -- cycle;
        \filldraw[color=red, fill=red!9, thick] (-3., 8.) -- (0.,8.) -- (0.,6.) -- cycle;
    
        \draw (0., 8.) node[right] {\footnotesize $T$};
        \draw (-2., 6.67) node[] {\footnotesize $\mathcal{D}_b$};
        \draw (-1., 7.33) node[] {\footnotesize ${\mathcal{D}_c}$};
        \draw (-1.5, 3.) node[] {\footnotesize $\mathcal{D}_{a}$};
        
        \draw [color=black, dashed] (0.,4.) -- (-3.,4.);
        \draw (0., 4.) node[right] {\footnotesize $T-2d$};
    \end{tikzpicture}
   \end{minipage}\hfill
   \begin{minipage}{0.48\textwidth}
     \centering
     \begin{tikzpicture}[scale=0.58] 

        \pgfmathsetmacro\AngleFuite{135}
        \pgfmathsetmacro\coeffReduc{.8}
        \pgfmathsetmacro\clen{2}
        \pgfmathsinandcos\sint\cost{\AngleFuite}
        
        \begin{scope} [x     = {(\coeffReduc*\cost,-\coeffReduc*\sint)},
                       y     = {(1cm,0cm)}, 
                       z     = {(0cm,1cm)}]

        \draw[thick,-stealth,black]  (0,0,0)  -- (0,0,9) node[above] {$t$};
        \draw[thick,-stealth,black]   (0,0,0)  -- (3,0,0) node[below left] {$s$};
        \draw[thick,-stealth,black] (0,0,0)  -- (0,3,0) node[right] {$r$};
        \coordinate (O) at (0,0,0); 
        \coordinate (A) at (0,2,0); 
        \coordinate (B) at (2,2,0); 
        \coordinate (C) at (2,0,0); 
        
        \coordinate (H) at (0,0,8); 
        \coordinate (I) at (2,0,8); 
        \coordinate (J) at (2,2,8); 
        \coordinate (K) at (0,2,8); 
        
        \coordinate (D) at (0,0,5.2); 
        \coordinate (E) at (2,0,5.2); 
        \coordinate (F) at (2,2,5.2); 
        \coordinate (G) at (0,2,5.2); 
        
        \coordinate (L) at (0,0,3); 
        \coordinate (M) at (2,0,3); 
        \coordinate (N) at (2,2,3); 
        \coordinate (P) at (0,2,3);

        \draw[black,fill=black!9,opacity=0.8, thick] (D) -- (E) -- (F) -- (G) -- cycle;
        
        \draw[black,fill=black!9,opacity=0.8, thick] (C) -- (B) -- (F) -- (E) -- cycle;
        \draw[black,fill=black!9,opacity=0.8, thick] (B) -- (A) -- (G) -- (F) -- cycle;
        \draw[color=black, dashed] (L) -- (M) -- (N) -- (P) -- cycle;
        
        \draw[blue, fill=blue!9, opacity=0.8, thick] (J) -- (E) -- (F);
        \draw[blue, fill=blue!9, opacity=0.8, thick] (J) -- (F) -- (G);
        
        \draw[red,  fill=red!9, opacity=0.8, thick] (H) -- (K) -- (J) -- (I) -- cycle;
        \draw[red, fill=red!9 ,opacity=0.8] (I) -- (E) -- (J) -- cycle;
        \draw[red, opacity=0.8] (D) -- (E) -- (J) -- cycle;
        \draw[red, opacity=0.8] (G) -- (D) -- (J) -- cycle;
        \draw[red, opacity=0.8, thick] (E) -- (J);
        \draw[red, opacity=0.8, thick] (G) -- (J);
        \draw[red, fill=red!9 ] (J) -- (G) -- (K) -- cycle;
        
        \draw[red, thick] (I) -- (J);
        \draw[red, thick] (J) -- (K);
        \draw[black, thick] (E) -- (F);
        \draw[black, thick] (F) -- (G);
        \draw[red, thick] (I) -- (E);
        \draw[red, thick] (K) -- (G);
        \draw[red] (J) -- (D);
        \draw[blue, thick] (J) -- (F);
        \draw[blue] (J) -- (D);
        \draw[blue] (E) -- (D);
        \draw[blue] (G) -- (D);

        \draw (C) node[above left] {\footnotesize $-d$};
        \draw (A) node[above right] {\footnotesize $-d$};
        \draw (O) node[above left] {$0$};
        \draw (I) node[left] {\footnotesize $T$};
        \draw (E) node[left] {\footnotesize $T-d$};
        \draw (M) node[left] {\footnotesize $T-2d$};
        \draw[black] (1.5, 0.78, 2) node[] {\footnotesize $\mathcal{D}_{a}$};
        \draw[black] (0., 0.5, 7.5) node[] {\footnotesize ${\mathcal{D}_c}$};
        \draw[black] (1.5, 2.2, 6.5) node[] {\footnotesize $\mathcal{D}_b$};
      \end{scope}  
    \end{tikzpicture}
  \end{minipage}
  \caption{Left: Cross section of $\Dcal$ along $r=0$. Right: full domain $\Dcal = \Dcal_a \cup \Dcal_b \cup \Dcal_c $. } \label{fig:domain}
\end{figure}
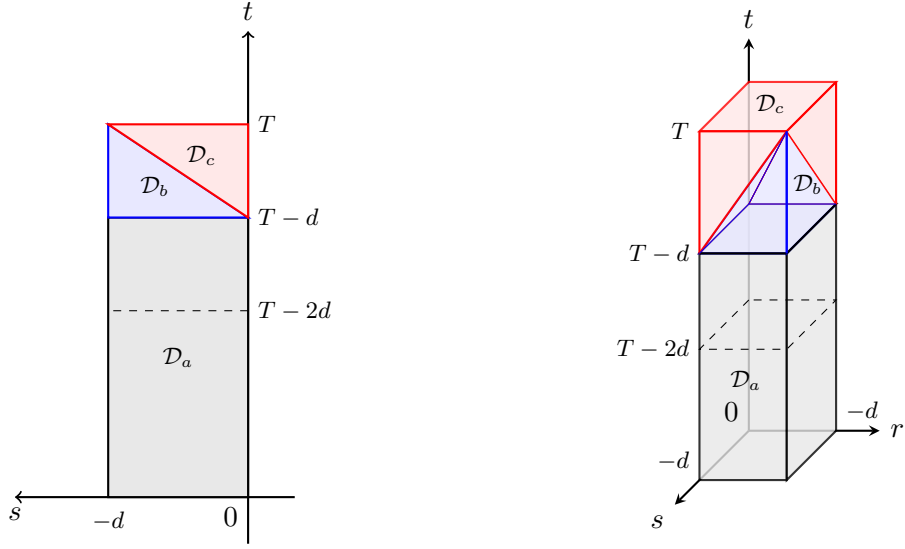

In the next section, we provide a proof of the existence and uniqueness of system \eqref{eq:a}-\eqref{eq:b}-\eqref{eq:c}, and a verification theorem yielding rigorously the optimal control and value of \eqref{eq:dynamic_heuri}-\eqref{eq:cost_functional_heuri}.

\section{Verification {and existence} results}
\label{section:verification_simplified}
In this section, we establish a verification result for the optimization problem \eqref{eq:dynamic_heuri}-\eqref{eq:cost_functional_heuri}.

\begin{theorem}[Verification Theorem]
\label{T:verif}
Assume that 
\begin{enumerate}
    \item \label{T:verif:i_simplified}  There exists a solution $P$ to \eqref{eq:a}-\eqref{eq:b}-\eqref{eq:c} in the sense of Definition  \ref{def:sol_E_i_simplified}, 
    \item \label{T:verif:ii_simplified} The control strategy defined as 
    \bes{
    \label{eq:optimal_control_final_simplified}
    \c^*_t =&\frac{-1_{t \leq T-d}  }{P_{\hat{22}}(t,0)} \bigg\{ X^{\c^*}_t P_{11}(t,0) + \int_{t-d}^t  P_{22}(t, 0, s-t) \c^*_s ds \bigg\}.
}
where $X^{\c^*}$ denotes the controlled state is an admissible control.
\end{enumerate}
Then the optimization problem \eqref{eq:dynamic_heuri}-\eqref{eq:cost_functional_heuri} admits \eqref{eq:optimal_control_final_simplified} as an optimal feedback control. Furthermore, for $z=(x,\gamma) \in H$, the value is given by \bes{
    \label{eq:value}
    V(z) &= P_{11}(0) x^2 +  2 x  \int_{-d}^0 P_{12}(0, s)  \gamma_s ds   + \int_{-d}^0  P_{\hat{22}}(0, s) \gamma^2_s ds \\
    &\;\;\;\;+\int_{[-d, 0]^2} \gamma_s\gamma_u P_{22}(0, s, r)  ds dr
    \\
    &= \langle P_0 z, z \rangle_H.
}
\end{theorem}
\begin{proof}
    The proof is a basic application of the martingale optimality principle, see \cite{el1981aspects}.    
   Let $\c \in \Acal$ and define 
   \bes{
   \label{eq:value_}
   V^\c_t &= P_{11}(t) (X_t^\c)^2 +  2 x  \int_{t-d}^t P_{12}(t, s-t)  \c_s ds   + \int_{t-d}^t  P_{\hat{22}}(t, s-t) \c^2_s ds \\
    &\;\;\;\;+\int_{[t-d, t]^2}  P_{22}(t, s-t, r-t) \c_s\c_r ds dr\\
    &=\langle P_t Z^\c_t, Z_t^\c \rangle_H.
   }
    An application of Itô's formula to \eqref{eq:value_} combined with differentiation under the integral symbol, authorized by the assumed boundedness of $P$ and its derivatives, yield
   \bes{
	d V_t^\c =& \Bigg\{\textbf{1}_t (X^\c_t)^2 +2 \bigg( \textbf{2}_t X_t^\c\c_t+ \textbf{3}_t X_t^\c\c_\td  +  X_t^\c \intd \textbf{4}_t(s) \c_s ds \bigg) \\
	&+ \textbf{5}_t \c_\td^2  +  \c_\td \textbf{6}_t(s) \intd \c_s ds + \c_t \intd \textbf{7}_t(s) \c_s ds \\
	&+   \intd \textbf{8}_t(s) \c^2_s ds + \intdd \textbf{9}_t(s,u) \c_s \c_u dsdu\Bigg\}dt + Z^\c_t dW_t,
	}
	where we have set 
	\begin{align}
	    &\textbf{1}_t = \dot{P}_{11}(t) && \textbf{2}_t = P_{12}(t,0) 
	    \\
	    &\textbf{3}_t = 2b P_{11}(t) - P_{12}(t,-d) && \textbf{4}_t(s) =(\partial_t-\partial_s)(E_2)(t,s-t) 
	    \\
	    &\textbf{5}_t = \sigma^2 P_{11}(t) - P_{\hat{22}}(t,-d) && \textbf{6}_t(s) =\left( b P_{12}(t,\cdot) - P_{22}(t,\cdot, -d) \right) (s-t)
	    \\
	    &\textbf{7}_t(s) = P_{22}(t,0, s-t)  && \textbf{8}_t(s) = (\partial_t-\partial_s)(P_{\hat{22}})(t,s-t) 
	    \\
	    &\textbf{9}_t(s,r) = (\partial_t-\partial_s-\partial_r)(P_{22})(t,s-t, r-t) ,
	\end{align}
	and 
	\bes{
	    Z^\c_t = 2 \sigma {\c}_{t-d} \left(X_t^{\c} P_{11}(t) +\intd \c_s P_{12}(t,s-t) ds  \right).
	}
	Then, using the set of constraints \eqref{eq:a}-\eqref{eq:b}-\eqref{eq:c} together with \eqref{eq:optimal_control_final_simplified} and a completion of the square in $\c$ yield
	\bes{
	    dV^\c_t = \left(P_{\hat{22}}(t, 0)\left(\c_t - \Tcal(\c)_t\right)^2\right) dt +  Z^\c_t dW_t,
	}
	where $\Tcal(\c)$ is defined as 
	\bes{
	\Tcal(\c)_t =&-\frac{1_{t \leq T-d}}{P_{\hat{22}}(t,0)} \bigg\{ X^{\c}_t P_{12}(t,0) + \int_{t-d}^t \c_s P_{22}(t,0, s-t) ds \bigg\}, \qquad t \leq T.
	}
	Note that since the kernels $P_i$'s are bounded and the control $\c^*$ is assumed to be admissible, $\c^* \in L_{\F}^{2}([0,T], \R)$, $X^\c$ is continuous and the stochastic integral $\int_0^. Z_s^\c dW_s$ is well posed.
	Furthermore it is a local martingale. Thus, there exists a localizing increasing sequence of stopping times $\{\tau_k\}_{k\geq 1}$ converging to $T$ such that $\int_0^{.\wedge \tau_k} Z_s^\c dW_s$ is a martingale for every $k \geq 1$. Then, for any $k\geq 1$
	\bes{
	   \E \left[ V^\c_{T \wedge \tau_k} \right] =& V_0^\c + \E \left[ \int_0^{T\wedge \tau_k} P_{\hat{22}}(s, 0)(\c_s - \Tcal(\c)_s)^2ds\right].
	}
	 Note that $t \mapsto V_t^\c$ is continuous since $P$ is bounded and $\c \in L^2_\mathbb{F}([-d, T], \R)$. Thus, an application of the dominated convergence theorem on the left term (recall that $\E\left[\sup_{t\leq T}|X_t^\c|^2\right]< \infty$, $\c \in L_{\mathbb{F}}^{2}([-d,T], \R)$ as $\c \in \mathcal{A}$) combined with the monotone convergence theorem on the right term yields, as $k \to \infty$

	\bes{
	   \E \left[V_{T}^\c \right] = \E[(X_T^\c)^2]  =  V_0^\c + \E \left[ \int_0^{T} P_{\hat{22}}(s, 0)(\c_s - \Tcal(\c)_s)^2ds\right].
	}
	Note that here we used the assumption $P_{\hat{22}}(t, 0) \geq 0$ on [0,T]. Since $P_{\hat{22}}$ is non-negative, we obtain that the optimal strategy is given by $\c^*$ and that the optimal value equals \eqref{eq:value}.
\end{proof}

\begin{proposition}
    \label{prop:admissibility}
    Assume that there exists a bounded 4-uplets $P$ solution to \eqref{eq:a}-\eqref{eq:b}-\eqref{eq:c} in the sense of Definition \ref{def:sol_E_i_simplified}. Then \eqref{eq:optimal_control_final_simplified} defines an admissible control. 
\end{proposition}
\begin{proof}
  Let $\gamma \in L^2([-d,0], \R)$. To prove the claim, note that it suffices to show that the equation
  
  \bec{
    \label{eq:alpha_pt_fix}
    \c_t &=\frac{-1_{t \leq T-d}}{P_{\hat{22}}(t,0)} \bigg\{ P_{12}(t,0) \left(x + \int_0^t \c_{s-d}(b ds + \sigma dW_s) \right) + \int_{t-d}^t \c_s P_{22}(t, 0, s-t) ds \bigg\},
    \\
    \c_s &= \gamma_s, \qquad s \in [-d, 0],
  }
  admits a solution in $L_{\F}^{2}([0,T], \R)$ and that the process $X$, then defined as 
  \bes{
  \label{eq:X}
  X_t = x + \int_0^t \c_{s-d}(b ds + \sigma dW_s), \qquad t \leq T,
  }
  satisfies $\E\left[\sup_{t\leq T}|X_t|^2\right]< \infty$. To prove the first point, consider the linear operator $\phi$ on $L_{\F}^{2}([0,T], \R)$ defined as, for any $a \in L_{\F}^{2}([0,T], \R)$
  \bes{
    \phi(a)_t = \frac{-1}{P_{\hat{22}}(t,0)} \bigg\{ P_{12}(t,0) \left(x + \int_0^t \hat a_{s-d}(b ds + \sigma dW_s) \right) + \int_{t-d}^t \hat a_s P_{22}(t, 0, s-t) ds \bigg\},
  }
  where $\hat a_t = 1_{t \leq 0} \gamma_t + 1_{t > 0} a_t$. For $\lambda \leq 0$, we endow $L_{\F}^{2}([0,T], \R)$ with the norm $\|a\|_{2,\lambda} = \sqrt{ \int_0^T e^{-\lambda s} |a_s|^2 ds}$. Then, for any $a,  a' \in L_{\F}^{2}([0,T], \R)$, we have 
  \bes{
  \label{eq:pt_fix}
   \|\phi(a) - \phi( a')\|^2_{2,\lambda} =& \E \left[\int_0^T e^{-\lambda s}|\phi(a)_s - \phi( a')_s|^2 ds \right] 
   \\
   \leq & 2 (\textbf{I} + \textbf{II}).
  }
An application of Jensen’s inequality on the normalised measure $\frac{dr}{s}$ on $[0,s]$, combined with $\frac{1 - e^{-\lambda(T-r)}}{\lambda} \leq (1 \vee T ) (1 \wedge \lambda^{-1})  $, and the Burkholder-Davis-Gundy inequality lead to 
\bes{
    \textbf{I} \leq & \E \left[ \int_0^T e^{-\lambda s} \left| \frac{ P_{12}(s,0) }{P_{\hat{22}}(s,0)} \int_0^s ( \hat a_{r-d}- \hat a_{r-d}')(b dr + \sigma dW_r)  \right|^2  ds \right]
    \\
    \leq & \sup_{s \leq  T} \left| \frac{P_{12}(s,0) }{P_{\hat{22}}(s,0)} \right|^2 \Bigg\{ \E \left[\int_0^T e^{-\lambda s} \left( \int_0^s (\hat a_{r-d}- \hat a'_{r-d}) b dr\right)^2 ds\right] \\
    & + \E \left[\int_0^T e^{-\lambda s} \left( \int_0^s (\hat a_{r-d}- \hat a'_{r-d}) \sigma dWr\right)^2 ds\right] \Bigg\}
    \\
    \leq & c  (1 \wedge \lambda^{-1}) \E \left[ \int_0^T e^{-\lambda r } (\hat a_{r-d}- \hat a'_{r-d})^2  dr \right],
}
where $c >0$ depends only on $b, \sigma, T$ and $\sup_{s \leq T} \left| \frac{P_{12}(s,0) }{P_{\hat{22}}(s,0)} \right|$. Furthermore, we have
\bes{
    \textbf{II} \leq & \E \left[ \int_0^T e^{-\lambda s}  \left|  \frac{1}{P_{\hat{22}}(s,0)}  \int_{s-d}^s (\hat a_r - \hat a'_r) P_{22}(r, 0, s-r) dr  \right|^2  ds  \right] 
    \\
    \leq & \sup_{\substack{s \leq T\\ r \in [-d,0]}} \left| \frac{P_{22}(s,0, r) }{P_{\hat{22}}(s,0)} \right|^2
    \E \left[ \int_0^T e^{-\lambda s} \left|  \int_0^s (\hat a_r - \hat a'_r) dr \right|^2 ds \right] 
    \\
    \leq & \hat c  (1 \wedge \lambda^{-1}) \E \left[ \int_0^T e^{-\lambda r } (a_{r}- \hat a'_{r})^2  dr \right],
}
where $\hat c >0$ depends only on $T$ and $\sup_{\substack{s \leq T \\ r \in [-d,0]} } \left| \frac{P_{22}(s,0, r) }{P_{\hat{22}}(s,0)} \right|$. Consequently, \eqref{eq:pt_fix} reduces to
\bes{
     \|\phi(a) - \phi(a')\|^2_{2,\lambda} \leq  & (c + \hat{c}) (1 \wedge \lambda^{-1}) \| a -  a' \|^2_{2, \lambda}.
}
As a result, for $\lambda$ large enough, $\phi$ is a contraction on the Banach space  $\left( L_{\F}^{2}([0,T], \R), \| . \|_{2,\lambda} \right)$, thus proving the existence of $\c \in L_{\F}^{2}([0,T], \R) $ solution to \eqref{eq:alpha_pt_fix}. Finally, an application of Burkholder-Davis-Gundy's inequality to \eqref{eq:X} yields $\E \left[ \sup_{t \leq T} |X_t|^2 \right]$. The proof is thus complete.
\end{proof}
Next, we give a sufficient condition for the existence of $P = (P_{11}, P_{12}, P_{\hat{22}}, P_{22})$ in terms of $b, \sigma, d$ and $T$. Let $a=(a_n)_{n \geq 1}$ be the sequence defined as 
\bec{
    \label{def:a_n}
    a_0 &= 1, \\
    a_{n+1} &= a_n - \frac{d}{a_n} \left( \frac{b}{\sigma } \right)^2, \qquad n \geq 0.
}
Let us denote $\Ncal : (d, b, \sigma)  \mapsto  \inf\{ n \geq 1 : a_n > 0 \text{ and } a_{n+1} \leq 0 \}$. Clearly, $\Ncal$ is a well defined finite valued function on $\R^3$ whose image is not restricted to $\{0\}$.
\begin{proposition}
\label{prop:ricatti}
Assume $\Ncal(d, b, \sigma) \geq 2$ and $T < \Ncal(d, b, \sigma)d$ . Then \eqref{eq:a}-\eqref{eq:b}-\eqref{eq:c} has a unique solution in the sense of definition \ref{def:sol_E_i_simplified} on $[0,T]$ with $0<a_{\Ncal(d, b, \sigma)} \leq P_{11}(0) < 1$. 
\end{proposition}
\begin{proof}
   See appendix \ref{appendix:fixed_point_banach}.
\end{proof}
\begin{remark}
    Note that when $d=0$, the sufficient condition above reduces to $\sigma \neq 0$.
\end{remark}

Let us give some intuition as of why the delay feature induces the condition on the coefficients described above to ensure existence. We focus on the first slice $[T-2d, T-d]\times [-d, 0]$ of the domain, where the solution $P$ is not trivial.
\begin{center}
\begin{minipage}{0.48\textwidth}
    First, note that since $P_{\hat{22}}$ is a transport of $P_{11}$ which takes the form $P_{\hat{22}}(t,s) = \sigma^2 P_{11}(t+s+d)1_{t+s+d \leq T}= \sigma^2 1_{t+s+d \leq T}$. On this slice, the kernel $P_{11}$ can be expressed in the following integral form
     \begin{equation}
     \label{eq:p11}
         P_{11}(t)=1-\left(\frac{b}{\sigma}\right)^2 \int_t^{T-d}\left({\blue\mathbf{ \frac{P_{12}(s,0)}{b} }}\right)^2 ds,
     \end{equation}
     for $t\in [T-2d, T-d]$. Looking at $P_{12}(\cdot,0)$, we have 
     \bes{
        P_{12}(t,0) &= b - \sigma^{-2} \int_{t}^{T-d} P_{12}(x,0)\\& \;\;\;\; \times \underbrace{P_{22}(x,t-x,0)}_{\lesssim b^2}dx, 
     }
     see also \eqref{eq:system_integral_form_first_slice}.
     On the right, we represent the value of the normalized kernel $P_{12}/b$ in the different areas of the domain $[T-2d, T]\times [-d, 0]$. If we visualize the evolution of the normalized kernel $P_{12}/b$ in a backward way on the slice $[T-2d, T-d]$, we see that this term is equal to a transport of its value on the boundary $\frac{P_{12}(T-d,s)}{b}=1$, represented by the blue arrows, minus the integral of a positive source term which is independent of $t\in [T-2d, T-d] \mapsto P_{11}(t)$. 
     
   \end{minipage}\hfill
   \begin{minipage}{0.48\textwidth}
        \centering
     \begin{tikzpicture}[scale=1]
        \draw[thick,-stealth,black, ->] (0.,-1.0) -- (0.,10.) node[above] {$t$}; 
        \draw[thick,-stealth,black, ->] (1.,0.) -- (-5.,0) node[below] {$s$};
        \draw (0., 0.) node[below left] {$0$};
    
        \draw (-3., -0.1) node[below] {\footnotesize $-d$};
        \draw (0., 6.) node[right] {\footnotesize $T-d$};
        \filldraw[color=black, fill=black!9] (0,0) rectangle (-3.,6.);
        \filldraw[color=blue, fill=blue!9, thick] (0.,6.) -- (-3.,8.) -- (-3.,6.) -- cycle;
        \filldraw[color=red, fill=red!9, thick] (-3., 8.) -- (0.,8.) -- (0.,6.) -- cycle;
    
        \draw (0., 8.) node[right] {\footnotesize $T$};
        \draw (-2., 6.67) node[] {\footnotesize {\blue $1$}};
        \draw (-1., 7.33) node[] {\footnotesize {\red $0$}};
        \draw (-1.5, 3.) node[] {\footnotesize $\mathcal{D}_{a}$};
        
        \draw [color=black, dashed] (0.,4.) -- (-3.,4.);
        \draw (0., 4.) node[right] {\footnotesize $T-2d$};
        
        \draw[blue, ->] (-3,6) -- (0,4);
        \draw[blue, ->] (-2.5,6) -- (0,4.33);
        \draw[blue, ->] (-2,6) -- (0,4.66);
        \draw[blue, ->] (-1.5,6) -- (0,4.99);
        \draw[blue, ->] (-1,6) -- (0,5.32);
        \draw[blue, ->] (-0.5,6) -- (0,5.65);
        \draw[thick,blue] (0,6.) -- (0,4.);
        
        \draw [decorate,decoration={brace,amplitude=4pt},xshift=0.5cm,yshift=0pt]
      (0,5.7) -- (0,4.3) node [midway,right,xshift=.1cm] { {\blue $\frac{\mathbf{P_{12}(t,0)}}{b} \lesssim 1$}};
      
    \end{tikzpicture}
\end{minipage}
\end{center}

Consequently, the delay $d>0$ makes the integral term in \eqref{eq:p11} independent of $t\in [T-2d, T-d] \mapsto P_{11}(t)$ and of the order of  $d\left(\frac{b}{\sigma}\right)^2$. If this quantity is too large, the kernel $P_{11}$ can then reach negative values, thus making $P_{\hat{22}}$ negative on the next slice $[T-3d, T-2d]$ and therefore preventing the system \eqref{eq:a}-\eqref{eq:b}-\eqref{eq:c} from having a solution. Repeating this argument from slice to slice of size $d$ in a backward manner induces the aforementioned sufficient condition. Note that these arguments break down when $d=0$. These arguments are precisely developed in Appendix \ref{appendix:fixed_point_banach}.

\section{Deep learning scheme}
\label{section:nn}
\subsection{A quick reminder of PINNs and Deep Galerkin method for PDEs}

In order to solve \eqref{eq:a}-\eqref{eq:b}-\eqref{eq:c}, we will make use of neural networks in the spirit of the emerging Physics Informed Neural Networks (PINNs) and Deep Galerkin literatures, see \cite{sirignano2018dgm} and \cite{raissi2019physics} to name just a few. We first recall some of the main ideas. Assume we have a nonlinear partial differential equation of the form 
\bes{
\label{eq:pde}
    \partial_t u + \Ncal(u) &= 0, \qquad \text{on } \Omega, \\
    u &= g, \qquad \text{on } \partial \Omega,
}
where $\Ncal$ is a nonlinear operator, $\Omega$ a bounded open subset and $g$ a function on the boundary of the domain. The  main idea is to approximate the solution $u$ to \eqref{eq:pde} by a deep neural network. Let us call $t \mapsto  u(t, \Theta)$ this network, where $\Theta$ and $t$ denote respectively its parameters and a generic element of $\Omega \cup \partial \Omega$. The goal is to find a $\Theta$ so that $t \mapsto  u(t, \Theta)$ satisfies \eqref{eq:pde}.  To do so, the idea is to proceed by minimizing the mean square error loss 
\bes{
    \Lcal(\Theta,  \Tcal) &= \Lcal_u(\Theta,  \Tcal) + \Lcal_f(\Theta,  \Tcal),
    \\
    \Lcal_u(\Theta,  \Tcal) &= \frac{1}{|\Tcal|} \sum_{t\in \Tcal} |(\partial_t + \Ncal){u}(t, \Theta) |^2 1_{t \in \Omega},
    \\
    \Lcal_f(\Theta,  \Tcal) &= \frac{1}{|\Tcal|} \sum_{t\in \Tcal} |{u}(t, \Theta) - g(t)|^2 1_{t \in \partial \Omega},
}
where $\Lcal_f$ is the loss associated with the initial and boundary constraints on $\partial \Omega$, $\Lcal_u$ the loss associated to the PDE constraint $\partial_t u + \Ncal(u) = 0$ on $\Omega$ and $\Tcal$ a random subset of $\partial \Omega \cup \Omega$.

\subsection{A tailor-made algorithm for Riccati partial differential equations }
Although \eqref{eq:a}-\eqref{eq:b}-\eqref{eq:c} naturally fits the framework of PINNs, we make use of the structure exhibited on the operator $P$ to build a tailor-made algorithm to approximate the system of Riccati transport PDEs \eqref{eq:a}-\eqref{eq:b}-\eqref{eq:c}. 

\textit{Step 1:} We define one neural network for each kernel $P_{11}, P_{12}, P_{\hat{22}}, P_{22}$, as described in Figure \ref{fig:nn_structure}. Note that usually a unique neural network is used as a surrogate to the function that is to be approximated. 
\begin{figure}[h]
\centering
        \includegraphics[width=.8\textwidth]{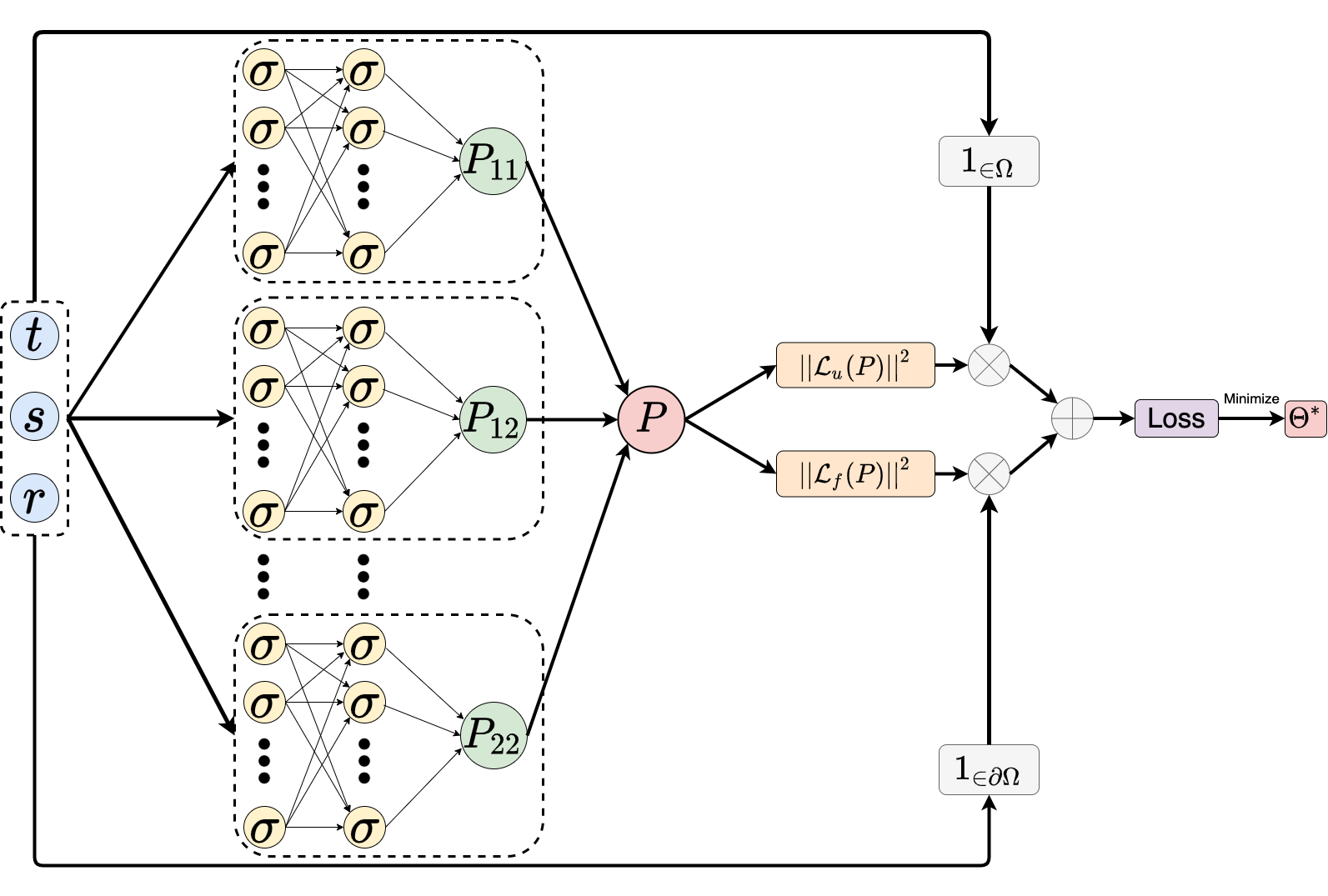}
         \caption{Structure of the model used to solve \eqref{eq:a}-\eqref{eq:b}-\eqref{eq:c}.}
    \label{fig:nn_structure}
\end{figure}

\textit{Step 2}: We build specific loss functionals for each of our neural networks. To illustrate this, consider for instance the constraint imposed on the derivative of  $t \mapsto P_{11}(t)$:
\bes{
\label{eq:constain_E_1}
    \dot{P}_{11}(t) &= \frac{P_{12}(t, 0)^2}{P_{\hat{22}}(t, 0)}, \qquad t \in [0,T].
}
Note here that, as in the  previous sections, we use the convention $0^2/0=0$. As a result \eqref{eq:constain_E_1} can be rewritten as 
\bes{
    \dot{P}_{11}(t) &= \frac{P_{12}(t, 0)^2}{P_{\hat{22}}(t, 0)}, \qquad t \leq T-d,\\
    \dot{P}_{11}(t) &= 0, \qquad t> T-d.
}
Thus, a natural contribution to the total loss function to enforce \eqref{eq:constain_E_1} would be 
\bes{
    \label{eq:sub_loss_1}
    \Lcal_{11}(\Theta_{11}, \Theta_{12}, \Theta_{\hat{22}} , \Tcal ) = \frac{1}{|\Tcal |} \sum_{t \in \Tcal } \left(  \partial_t {P}_{11}(\Theta_{11}, t)- \frac{ P_{12}(\Theta_{12},t, 0)^2}{P_{\hat{22}}(\Theta_{\hat{22}}, t, 0)}  \right)^2, \qquad \Tcal  \subset [0,T],
}
and the natural gradient descent step associated to the constraint \eqref{eq:constain_E_1} would be
\bes{
    \Theta_i \leftarrow& \Theta_i - \eta  \nabla_{\Theta_i} \Lcal_{11}(\Theta_{11}, \Theta_{12}, \Theta_{\hat{22}} , \Tcal), \qquad i \in \{11, 12, \hat{22}\},
}
 Consequently, the constraint \eqref{eq:constain_E_1} a priori entails the updating of $P_{11}, P_{12}$ and $P_{\hat{22}}$. In particular, it requires to compute the gradient of $(\Theta_{12}, \Theta_{\hat{22}}) \mapsto \sum_{t \in \Tcal} \frac{ P_{12}(\Theta_{12},t, 0)^2}{ P_{\hat{22}}(\Theta_{\hat{22}}, t, 0)}$ which is expected to be highly unstable as $t,s \mapsto P_{\hat{22}}(t,s)$ vanishes for $t+s \geq T-d$. To mitigate this issue, the term $t \mapsto \frac{ P_{12}(\Theta_{12},t, 0)^2}{P_{\hat{22}}(\Theta_{\hat{22}}, t, 0)}$ is considered as an exogenous source term for $P_{11}$ which is fixed when we train $P_{11}$. 
\bes{
    \dot{P}_{11}(t) &= \underbrace{\frac{P_{12}(t, 0)^2}{P_{\hat{22}}(t, 0)}}_{\text{Seen as a fixed exogenous source term when $P_{11}$ is trained}} \qquad, t \in [0,T].
}
To implement this idea, a second set of neural networks $(\tilde{P}_k( \tilde \Theta_k))_{k \in \{11,12,\hat{22}, 22\}}$ is initialized with $ \tilde \Theta_k = \Theta_k$ for $k \in \{11,12,\hat{22}, 22\}$ at initialization. These additional networks are then used as surrogates to the right-hand side source terms and will not be used for the computation of the gradients of the losses. They will only be updated at the end of each batch training. Consequently, the gradient descent scheme implemented for each batch $\Tcal$ is the following:
\bes{
\label{gradients}
&\text{Step 1}  \left \{
    \begin{array}{ll}
       \Theta_{11} &\leftarrow \Theta_{11} - \eta  \nabla_{\Theta_{11}} \left( \Lcal_{11}^r(\Theta_{11}, \tilde \Theta_{12}, \tilde \Theta_{\hat{22}}, \Tcal) + \Lcal^{f}_1(\Theta_{11}, \Tcal)  \right) \\
    \Theta_{12} &\leftarrow \Theta_{12} - \eta \nabla_{\Theta_{12}} \left( \Lcal^r_{12}(\Theta_{12}, \tilde \Theta_{\hat{22}} ,\tilde \Theta_{22} , \Tcal) + \Lcal^b_{12}(\tilde \Theta_{11}, \Theta_{12}, \Tcal)  + \Lcal^f_{12}(\Theta_{12}, \Tcal)  \right) \\
    \Theta_{\hat{22}} &\leftarrow \Theta_{\hat{22}} - \eta \nabla_{\Theta_{\hat{22}}} \left( \Lcal^r_{{\hat{22}}}(\Theta_{\hat{22}} , \Tcal) + \Lcal^b_{\hat{22}}(\tilde \Theta_{11}, \Theta_{\hat{22}} , \Tcal) + \Lcal^f_{\hat{22}}(\Theta_{\hat{22}}, \Tcal) \right) \\
    \Theta_{22} &\leftarrow \Theta_{22} - \eta \nabla_{\Theta_{22}} \left( \Lcal_4^r(\Theta_{\hat{22}}, \Theta_{22}, \tilde \Theta_{22}, \Tcal) + \Lcal^b_{22}(\Theta_{12}, \Theta_{22}, \tilde \Theta_{22}, \Tcal) +\Lcal_{22}^f(\Theta_{22}, \Tcal) \right), \\
    \end{array}
\right. \\
&\text{Step 2} \left \{ \tilde \Theta_k \leftarrow \Theta_k, \qquad k \in \{11,12,\hat{22},22 \},\right.
}
where the $\Lcal_i^r$'s stand for the \textit{residual loss}, the $\Lcal_i^b$'s stand for the \textit{boundary loss} and the $\Lcal_i^f$'s stand for the \textit{final loss}. The precise definitions of the losses are given below.

For each neural network $ P_k(\Theta_k)$, $k\in \{11,12,\hat{22},22\}$, a \textit{follower network} is initialized $\tilde P_k( \tilde \Theta_k)$, $k\in \{11,12,\hat{22},22\}$. These follower networks serve as surrogate for the source terms in \eqref{eq:a}-\eqref{eq:b}-\eqref{eq:c}. They condition the loss functionals that are used to train the $P_k$'s and are updated at the end of each batch as described in Algorithm \ref{algo:nn_training}. 

\noindent\textbf{Losses of $P_{11}$:}
\bes{
\label{loss_1}
    \Lcal^r_{11}(\Theta_{11}, \tilde \Theta_{12}, \tilde \Theta_{\hat{22}}) &= \frac{1}{|\Tcal|} \sum_{t \in \Tcal }\left(\partial_t P_{11}(t, \Theta_{11})  - \frac{\tilde P_{12}(t,0, \tilde \Theta_{12})^2}{\tilde P_{\hat{22}}(t,0, \tilde \Theta_{\hat{22}})} \right)^2,
    \\
    \Lcal^{f}_{11}(\Theta_{11}, \Tcal) &= \frac{1}{|\Tcal|} \sum_{t \in \Tcal } \left( \left(P_{11}(t, \Theta_{11})-1 \right) ^2 1_{t=T} \right) .
    \\
}
\noindent\textbf{Losses of $P_{12}$:}
\bes{
\label{loss_2}
    \Lcal^r_{12}(\Theta_{12}, \tilde \Theta_{\hat{22}}, \tilde \Theta_{22}, \Tcal) &= \frac{1}{|\Tcal|} \sum_{t,s \in \Tcal }\left((\partial_t - \partial_s) P_{12}(t,s, \Theta_{12})  - \frac{\tilde P_{12}(t,0, \tilde \Theta_{12}) \tilde P_{22}(t, s,0, \tilde\Theta_{22})}{\tilde P_{\hat{22}}(t,0, \tilde \Theta_{\hat{22}})} \right)^2, 
    \\
     \Lcal^b_{12}(\tilde \Theta_{11}, \Theta_{12}, \Tcal) &= \frac{1}{|\Tcal|} \sum_{t,s \in \Tcal }\left( P_{12}(t,s, \Theta_{12})  -  b\tilde P_{11}(t, \tilde \Theta_{11}) \right)^2 1_{s=0, t\neq 0}  ,
    \\
    \Lcal^{f}_{12}(\Theta_{12}, \Tcal) &= \frac{1}{|\Tcal|} \sum_{t,s \in \Tcal } \left( P_{12}(t, s, \Theta_{12}) ^2 1_{t=T} \right) .
    \\
}
\noindent\textbf{Losses of $P_{\hat{22}}$:}
\bes{
\label{loss_3}
 \Lcal^r_{\hat{22}}(\Theta_{\hat{22}}, \Tcal) &= \frac{1}{|\Tcal|} \sum_{t,s \in \Tcal }\left((\partial_t - \partial_s) P_{\hat{22}}(t,s, \Theta_{\hat{22}}) \right)^2 ,
    \\
    \Lcal^b_{\hat{22}}(\tilde \Theta_{11}, \Theta_{\hat{22}}, \Tcal) &= \frac{1}{|\Tcal|} \sum_{t,s \in \Tcal }\left( P_{\hat{22}}(t,s, \Theta_{\hat{22}})  - \sigma^2\tilde P_{11}(t, \tilde \Theta_{11}) \right)^2 1_{s=0,  t \neq T}  ,
    \\
    \Lcal^{f}_{\hat{22}}(\Theta_{\hat{22}}, \Tcal) &= \frac{1}{|\Tcal|} \sum_{t,s \in \Tcal } \left( P_{\hat{22}}(t, s, \Theta_{\hat{22}}) ^2 1_{t=T} \right) .
    \\
}
\noindent\textbf{Losses of $P_{22}$:}
\bes{
\label{loss_4}
     \Lcal^r_{22}(\Theta_{22}, \tilde \Theta_{22},  \Tcal) &= \frac{1}{|\Tcal|} \sum_{t,s, r \in \Tcal }\left((\partial_t - \partial_s - \partial_r) P_4(t,s,r, \Theta_{22}) - \frac{\tilde P_{22}(t, 0, r,  \tilde \Theta_{22}) \tilde P_{22}(t, s, 0,  \tilde \Theta_{22})}{\tilde P_{\hat{22}}(t,0, \tilde \Theta_{\hat{22}})} \right)^2 ,
    \\
    \Lcal^b_{22}(\tilde \Theta_{12}, \Theta_{22}, \Tcal) &= \frac{1}{|\Tcal|} \sum_{t,s, r \in \Tcal }\left( P_{22}(t,s,r, \Theta_{22})  - b\tilde P_{22}(t,s, \tilde \Theta_{12}) \right)^2 1_{r=0,  t \neq T}  \\
    & \;\;\;\;\; +  \frac{1}{|\Tcal|} \sum_{t,s, r \in \Tcal }\left( P_{22}(t,s,r, \Theta_{22})  - b\tilde P_{22}(t,r, \tilde \Theta_{12}) \right)^2 1_{s=0,  t \neq T} ,
    \\
    \Lcal^{f}_{22}(\Theta_{22}, \Tcal) &= \frac{1}{|\Tcal|} \sum_{t,s,r \in \Tcal } \left( P_{22}(t, s,r,  \Theta_{22}) ^2 1_{t=T} \right).
    \\
}

\begin{algorithm}
\SetAlgoLined
\KwResult{A set of optimized parameters $\Theta^*=(\Theta_k^*)_{k \in \{11,12, \hat{22}, 22\}}$;}
 Initialize the learning rate $\eta$, the neural networks ${P}(\Theta) = (P_k(\Theta_k))_{k }$ and $\tilde{P}(\tilde \Theta) = (\tilde{P}_k(\tilde \Theta_k))_{k}$ \;
 Copy the weights $\tilde \Theta_k \leftarrow \Theta_k$, $k \in \{11,12, \hat{22}, 22\} $ \;
 \For{each batch}{
    Randomly sample $\Tcal \subset [0,T]\times  [-d, 0]^2$ \;
    Compute the gradient $\nabla_\Theta \mathcal{L}(\Theta,\tilde \Theta,\mathcal{T})$ as in \eqref{gradients} \;
    Update $\Theta \leftarrow \Theta - \eta \nabla_\Theta \mathcal{L}(\Theta, \tilde \Theta,\mathcal{T})$\;
    Update $\tilde \Theta \leftarrow \Theta$;
}
\textbf{Return:} The set of  optimized parameters  $\Theta^*$.
\caption{Deep learning scheme to solve \eqref{eq:a}-\eqref{eq:b}-\eqref{eq:c}}
\label{algo:nn_training}
\end{algorithm}

\section{Applications to mean-variance portfolio selection with execution delay}
\subsection{One asset with delay}
We now aim at solving the celebrated example of mean-variance portfolio selection, see \cite{markowitz1952portfolio}, with execution delay in the spirit of the problem of hedging of European options with execution delay presented in \cite{fabbri2014infinite}. We present here the settings. Let us consider a standard Black-Scholes financial market, composed of a risk-less asset with zero interest rate
\bes{
    \label{eq:bond_asset}
     S^0_t = 1, \quad t \in [0,T],
}
and a risky asset with dynamics
\bes{
    \label{eqrisky_asset}
    d S_t = S_t \left\{ \left( \sigma \lambda \right) dt  + \sigma dW_t  \right\},     \quad t \in [0,T],          
}
where $ \lambda$ and $ \sigma$ are constants representing respectively the risk premium and the volatility of the risky asset. At every time $t \in [0,T]$ the investor chooses, based on the information $\mathcal{F}_t$, to allocate the amount of money $\c_t \in \R$ into the risky asset. However, due to execution delays this order will be executed at time $t + d$. Set $N_t^\c$ (respectively $N^0_t$) the number of risky (respectively risk-less) shares held at time $t$.
Then, given an investment strategy $\c \in \mathcal{A}$, the value $\left(X_t^\c \right)_{t \in [0,T]}$ of the portfolio, that we suppose self-financing, follows the dynamics
\bes{
    d X_t^\c &= N_t^\c dS_t + \underbrace{(dN_t^\c)S_t + (d N^0_t)S^0_t}_{=0 \text{ , self-financing}} \\
            &= \underbrace{N_t S_t}_{\c_{t-d}} \left\{ ( \sigma \lambda) dt  + \sigma dW_t  \right\}.
}
Consequently, the controlled state equation of the portfolio's value is of the form
\bec{
    &dX_t^\c = \c_{t-d} \left(\left(\sigma \lambda\right)  dt +\sigma  dW_t\right), \qquad t \in [0,T],\\
    &X_0 = x_0, \quad \c_s = \gamma_s, \qquad  \forall s \in [-d, 0],
}
with $x_0 > 0$ and $\gamma \in L^2([-d, 0], \R)$. The Mean-Variance portfolio selection problem in continuous-time consists in solving the following constrained problem
\bec{
\label{optimization_problem_mv}
    &\min_{\substack{\c \in \Acal}} \V (X_T^\c) \\
    &\text{s.t. } \E[X_T^\c] = c . \\
}
It is well-known that problem \eqref{optimization_problem_mv} is equivalent to the following \textbf{max-min problem}, see \citet[Section 6.6.2]{pham2009continuous}
\bes{
    \label{outer_inner_optimization_pb}
    &\max_{\eta \in \R} \min_{\substack{\c \in \Acal}} \E\left[\left(X_T^\c - (c-\eta) \right)^2\right] - \eta^2. \\
}
Thus, solving problem \eqref{optimization_problem_mv} involves two steps. First, the internal minimization problem in terms of the Lagrange multiplier $\eta$ has to be solved. Second, the optimal value of $\eta$ for the external maximization problem has to be determined. Thus, with $\xi = c-\eta$, we first define the \textbf{Inner optimization problem:}

\bes{
    \label{eq:inner_mv}
    \min_{\alpha \in \mathcal A}\E\left[ \left(X^{\alpha}_T-\xi\right)^2\right].
}
Note that, by setting $\tilde{X}^\c = X^\c - \xi$, the inner problem \eqref{eq:inner_mv} fits into the delayed LQ control problem analysed in Section \ref{section:verification_simplified}. We first solve the inner optimization problem \eqref{eq:inner_mv} in the following lemma.

\begin{lemma}
\label{L:verification_mv}
Fix $\eta \in \R$ and $\xi = c-\eta$. Assume $T < d \Ncal(d, (\sigma \lambda), \sigma)$ and define $\c^*(\xi)$ as the investment strategy
\bes{
    \label{eq:optimal_control}
     \c^*_t(\xi) =&\frac{-1_{t \leq T-d}}{P_{\hat{22}}(t,0)} \bigg\{ (X^{\c^*}_t-\xi) P_{12}(t,0) + \int_{t-d}^t \c^*_s(\xi) P_{22}(t, 0, s-t) ds \bigg\},
}
where $P$ denotes the solution to \eqref{eq:a}-\eqref{eq:b}-\eqref{eq:c} in the sense of Definition \ref{def:sol_E_i_simplified}. Then, the  inner minimization problem \eqref{eq:inner_mv} admits $\c^*(\xi)$ as an admissible optimal feedback strategy and the optimal value is 
\bes{
\label{eq:value_mv}
    V_0(\xi) =& P_{11}(0) (x_0-\xi)^2  + R({x}_0-\xi, \gamma),\\
}
where $R(\gamma)$ denotes the cost associated to the initial investment strategy $\gamma$ on $[-d, 0]$
\bes{
    R(x, \gamma) =&  2  x \int_{-d}^0 \gamma_s P_{12}(0,s) ds + \int_{-d}^0 \gamma^2_s P_{\hat{22}}(0,s) ds + \int_{[-d,0]^2} \gamma_s\gamma_u P_{22}(0,s, r)  ds dr.
}
\end{lemma}
\begin{proof}
    First, note that Proposition \ref{prop:ricatti} yields the existence and uniqueness of a solution  $P$ to \eqref{eq:a}-\eqref{eq:b}-\eqref{eq:c}. Furthermore the admissibility of $\c^*(\xi)$ results from Proposition \ref{prop:admissibility}. For any $\c \in \mathcal{A}$, define $\Tilde{X}^{\alpha}_t = X_t^{\alpha} - \xi$. Then, by It\^o's formula we have  
\bec{
    &d\Tilde{X}_t^\c = \c_{t-d} \left( (\sigma \lambda)  dt +\sigma  dW_t \right), \qquad t \in [0,T],\\
    &\Tilde{X}_0 = x_0 - \xi , \quad \c_s = \gamma_s, \qquad  \forall s \in [-d, 0].
}
As a result, $\Tilde{X}^{\c}$ and $X^{\c}$ have the same dynamics and $\Tilde{X}^{\c}_T=X^{\c}_T-\xi$ so that problem  \eqref{eq:inner_mv} can be alternatively written as
\bes{
\min_{\alpha \in \mathcal A}\E\left[ \left(\Tilde{X}^{\alpha}_T\right)^2\right].
}
Thus, the optimality of $\c^*(\xi)$ and the value \eqref{eq:value_mv} are immediately given by the verification theorem \ref{T:verif}. 
\end{proof}

\begin{theorem}
\label{T:verif_}
Assume $T < d \Ncal(d, (\sigma \lambda), \sigma)$. Then, the optimal investment strategy for the maximization problem \eqref{optimization_problem_mv} is  given by $a^*(\xi^*)$ defined in \eqref{eq:optimal_control}
with $\xi^* = c-\eta^*$ and 
\bes{
\label{eq:eta_star}
     \eta^* &= \frac{K(\gamma) + P_{11}(0)(x_0 - c)}{1-P_{11}(0)},
     \\
     K(\gamma) &= \int_{-d}^0 \gamma_s P_{12}(0,s) ds.
} 
Furthermore, the value of \eqref{optimization_problem_mv} is 
\bes{
     \label{eq:value_final}
     \V (X_T^{\c^*}) &= \frac{P_{11}(0)}{1 - P_{11}(0)} \left(x_0 - c + K(\gamma) \right)^2 
     \\
     & \quad + \int_{-d}^0 \gamma^2_s P_{\hat{22}}(0,s) ds + \int_{[-d,0]^2} \gamma_s\gamma_u P_{22}(0,s, r)  ds dr.
}
\end{theorem}
\begin{proof}
     As $T < d \Ncal(d, (\sigma \lambda), \sigma)$, Proposition \ref{prop:ricatti} ensures the existence and uniqueness of a solution  $P$ to \eqref{eq:a}-\eqref{eq:b}-\eqref{eq:c}. From Lemma \ref{L:verification_mv} and \eqref{outer_inner_optimization_pb}, we have that the max-min problem \eqref{outer_inner_optimization_pb}, which is equivalent to \eqref{optimization_problem_mv}, reduces to 
    \bes{
    \label{outer_inner_optimization_pb_}
    &\max_{\eta \in \R} \Big\{ V_0(c-\eta) -\eta^2  \Big\} \\
    &=\max_{\eta \in \R} \Big\{ P_{11}(0) \left(x_0 - (c-\eta) \right)^2 + R({x}_0-(c-\eta), \gamma)  - \eta^2 \Big\}. \\
}
Furthermore, since $T < d \Ncal(d, (\sigma \lambda), \sigma)$, Proposition \ref{prop:ricatti} ensures $0<P_{11}(0)<1$ so that the maximization  problem is strictly concave. Consequently, $\eta^*$ given by \eqref{eq:eta_star} is the optimal parameter. Setting $\xi^*=c-\eta^*$ in  \eqref{eq:optimal_control} and \eqref{eq:value_mv} results in the optimality of $\alpha^*(\xi^*)$, and the optimal value \eqref{eq:value_final} for the mean-variance problem \eqref{optimization_problem_mv}.
\end{proof}
\begin{remark}
    In the absence of pre-investment strategy, $\gamma =0$, we recover the usual form of the \textit{efficient frontier formula}
    \bes{
    \V (X_T^{\c^*}) = \frac{P_{11}(0) }{1 - P_{11}(0)} \left(x_0 - c \right)^2.
    }
\end{remark} 

\vspace{5mm}
Our observations from the simulations are the following.\\

\textbf{Efficient frontier:} In Figure \ref{fig:efficient_frontier}, we plot the efficient frontier for different delays $d$. Note that the greater the delay, the greater is the variance. This could have been foreseen by observing that, when the initial control is set to 0, i.e. $\gamma = 0_{L^2}$, the value function takes the form $V(x,0_{L^2}) = P_{11}(0)x^2$, see \eqref{eq:value}. As the value function is clearly an increasing function of the delay, the terminal variance of the portfolio $\V (X_T^{\c^*}) = \frac{P_{11}(0) }{1 - P_{11}(0)} \left(x_0 - c \right)^2$ is also an increasing function of the delay.\\

\textbf{Destabilization effect :} In Figure \ref{fig:X_a_1}, we plot different scenarios of portfolio allocation. We observe a {destabilization} effect and a supplement of volatility induced by the delay feature. We also note the tendency to invest more aggressively for greater values of the delay, as the investor has less time to ensure that the promised yield is achieved. We propose the following interpretation: In the classical setting, where $d=0$, if $Y^*$ denotes the optimal portfolio value process, the optimal investment strategy is of the form $\c^*_t = -\frac{b}{\sigma^2}(Y^{\c^*}-\mu^*)$ for a certain constant $\mu^* > c$. It can then easily be shown that $Y^* \leq \mu^*$. Thus, the optimal strategy consists in aiming from below at a fixed target $\mu^*$. When $d>0$, the optimal control is composed of an additional \textit{inertial term}
\bes{
     \c^*_t(\xi^*) =&\frac{-1_{t \leq T-d}}{P_{\hat{22}}(t,0)} \bigg\{ \underbrace{(X^{\c^*}_t-\xi^*) P_{12}(t,0)}_{\text{Usual mean-reverting term} }  + \underbrace{\int_{t-d}^t \c^*_s(\xi) P_{22}(t, 0, s-t) ds}_{\text{New inertial term}}  \bigg\},
}
so that, contrary to the case where $d=0$, the optimal control does not cancel when the target $\xi^*$ is attained. Also, note that at every time $t$, the agent doesn't have any control on the near future from $t$ to $t+d$.\\

\textbf{Kernel $P$ :} In Figure \ref{eq:kernels}, we plot the kernels $P_{11}, P_{12}, P_{\hat{22}}, $ and $P_{22}$. Note the discontinuity between $\Dcal_b$ and $\Dcal_c$ also described in Figure \ref{fig:domain}.

\begin{figure}[h!]
\centering
{
    \includegraphics[width=70mm]{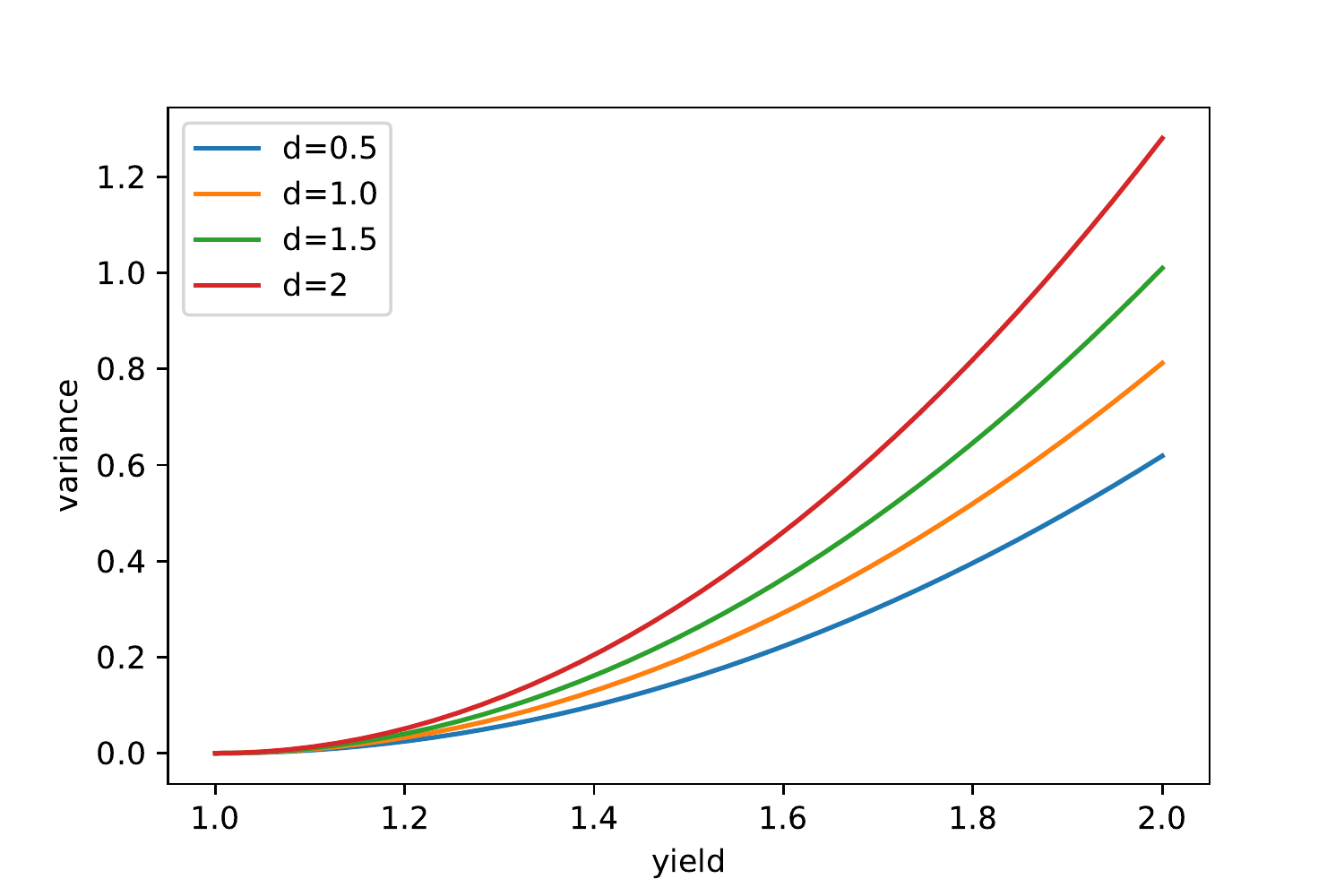}
    }
\caption{Efficient frontier with $\sigma=1$, $\lambda=0.5$, $T=5$ and $\gamma \equiv 0$.} 
\label{fig:efficient_frontier}
\end{figure}

\begin{figure}[h!]
\centering
\subfloat[]{
    \includegraphics[width=70mm]{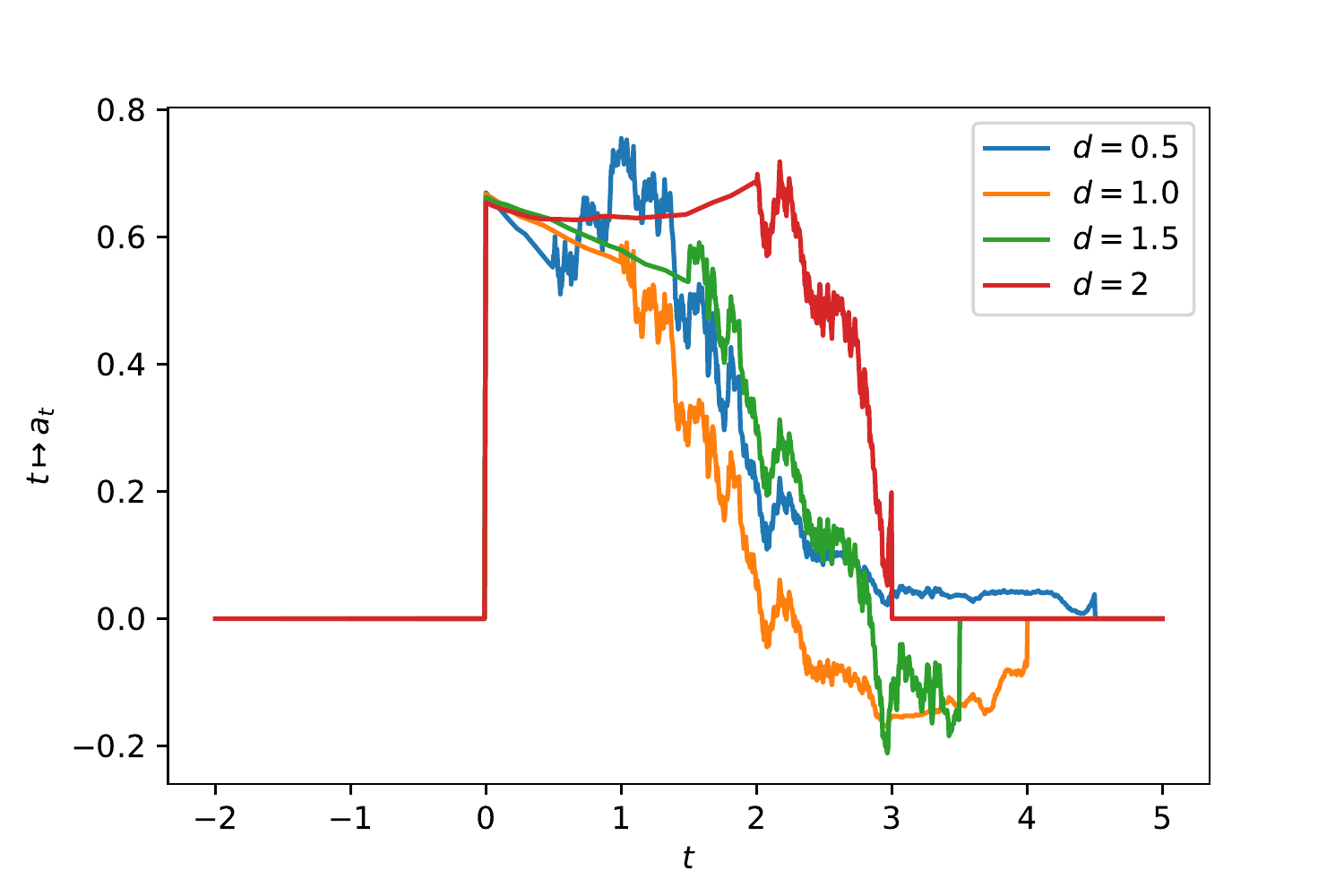}
    \includegraphics[width=70mm]{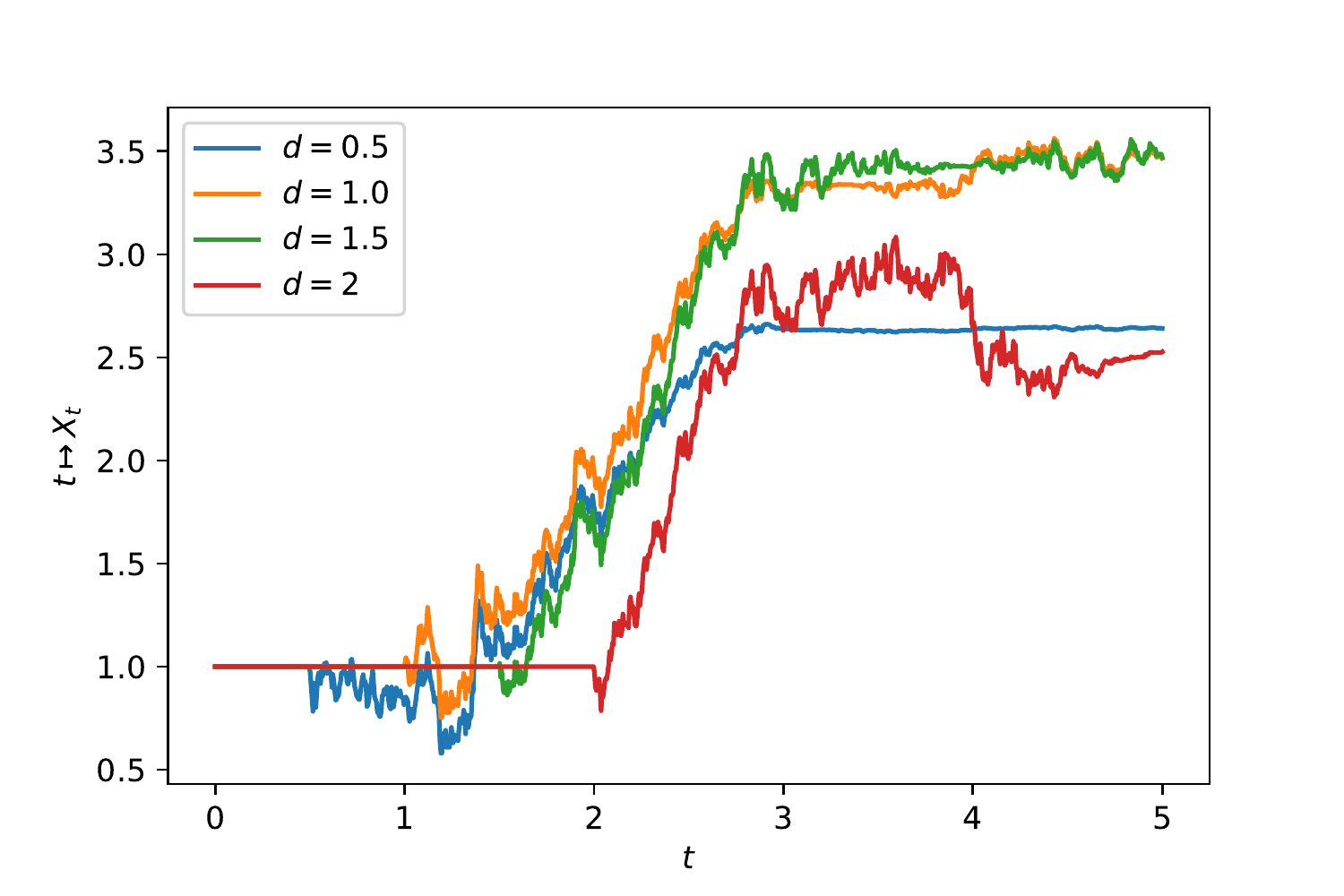}
    }
\vspace{-0.3cm}
\subfloat[]{
    \includegraphics[width=70mm]{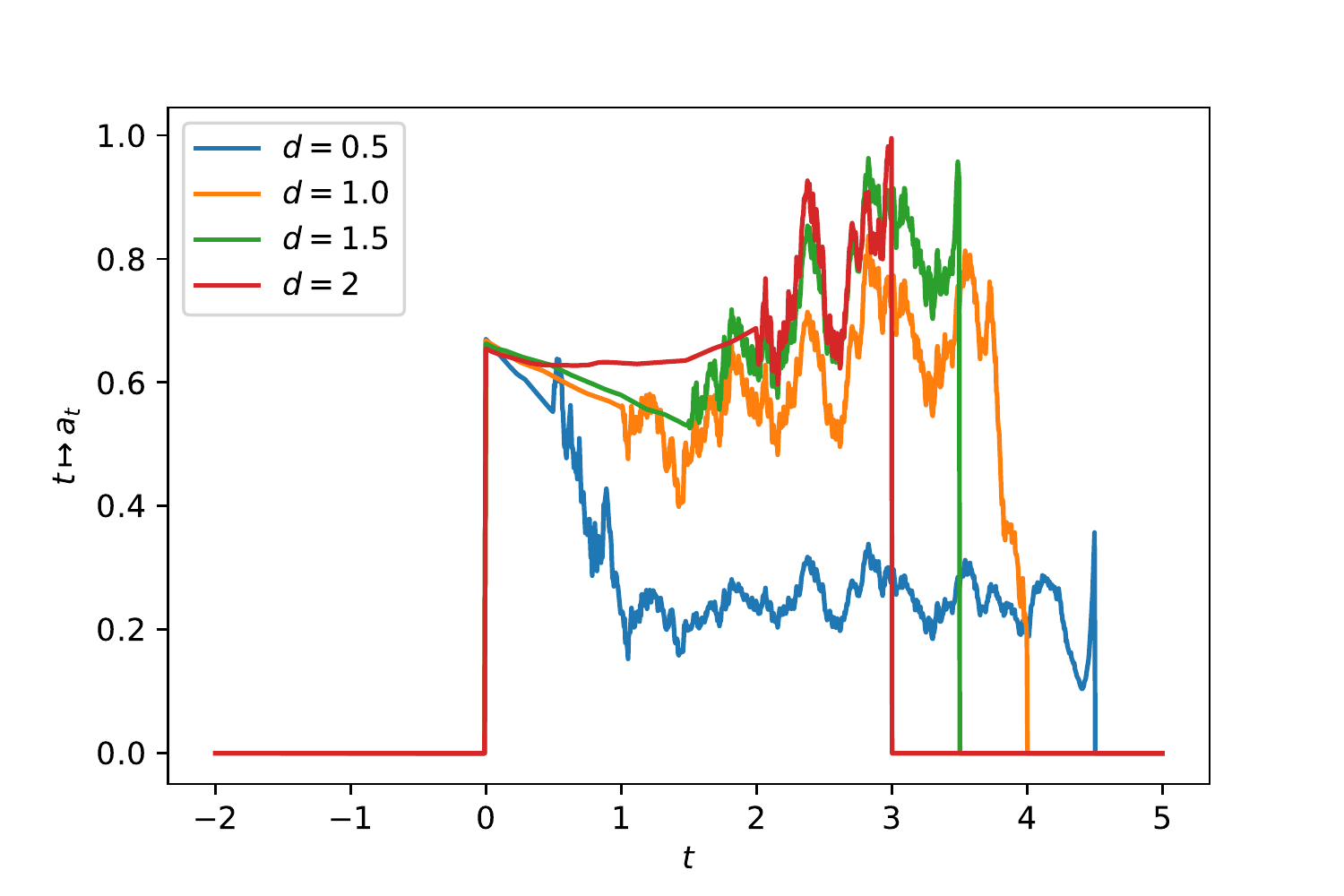}  
    \includegraphics[width=70mm]{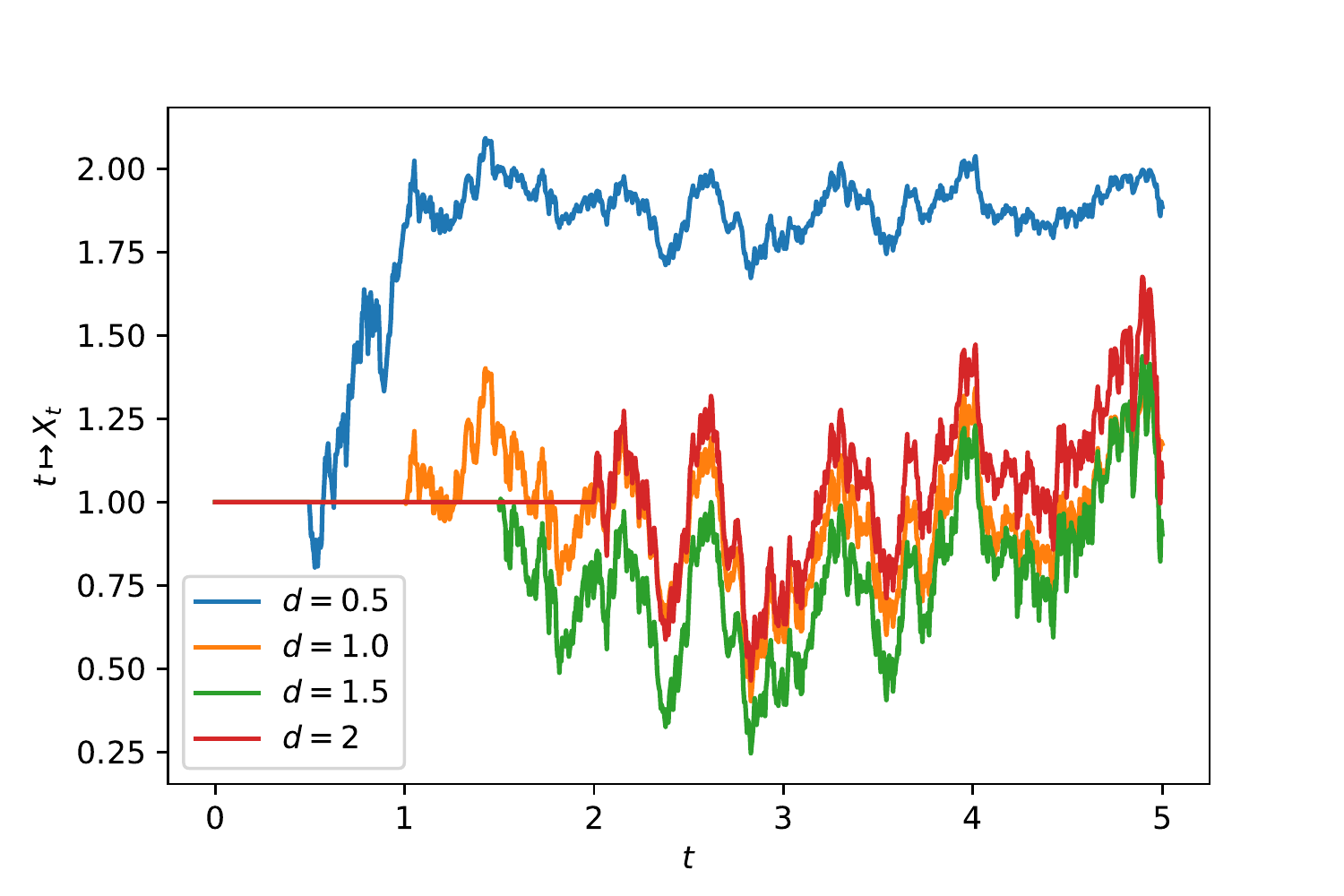}
    }
\vspace{-0.3cm}
\subfloat[]{
    \includegraphics[width=70mm]{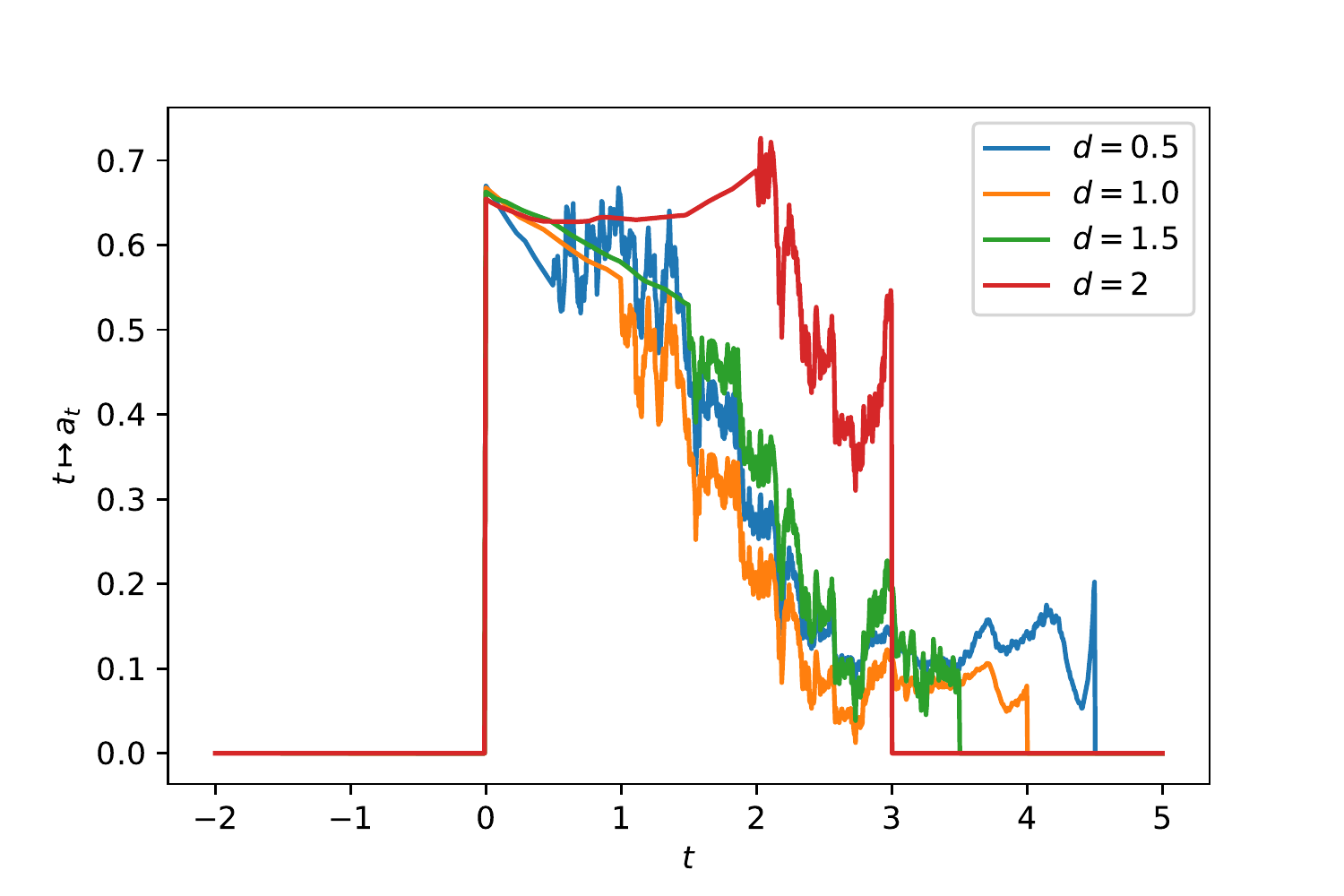}  
    \includegraphics[width=70mm]{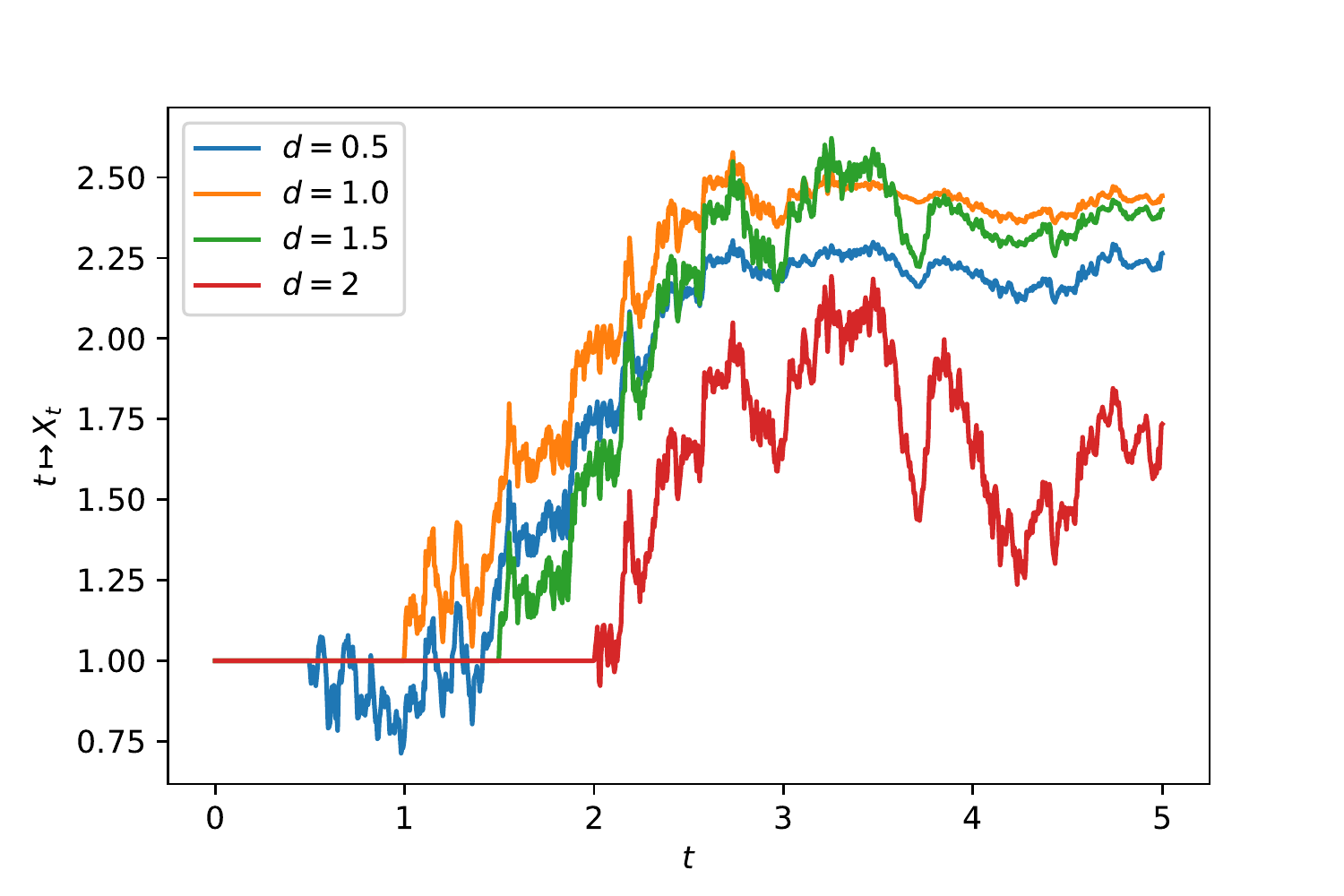}
    }
\caption{Different scenarios of the optimal portfolio with $c=1.6$, $\sigma=1$, $\lambda=0.5$, and $T=5$. Left: $t \mapsto \c^*$, right: $t\mapsto X_t^*$. Note the {destabilization} effect and the supplement of volatility induced by the delay feature. Note also the tendency to invest more aggressively the  delayed investor has, as she has less time to ensure the promised yield. $\xi^*(d=0.5)=2.57$, 
$\xi^*(d=1)=2.68$,
$\xi^*(d=1.5)=2.80$, $\xi^*(d=2)=2.97$.} 
\label{fig:X_a_1}
\end{figure}

\begin{figure}[h!]
\centering
\subfloat[Right: $t \mapsto P_{11}(t)$. Left: $t,s \mapsto P_{12}(t,s)$.]{
\label{fig:horizon_T_1}
    \includegraphics[width=70mm]{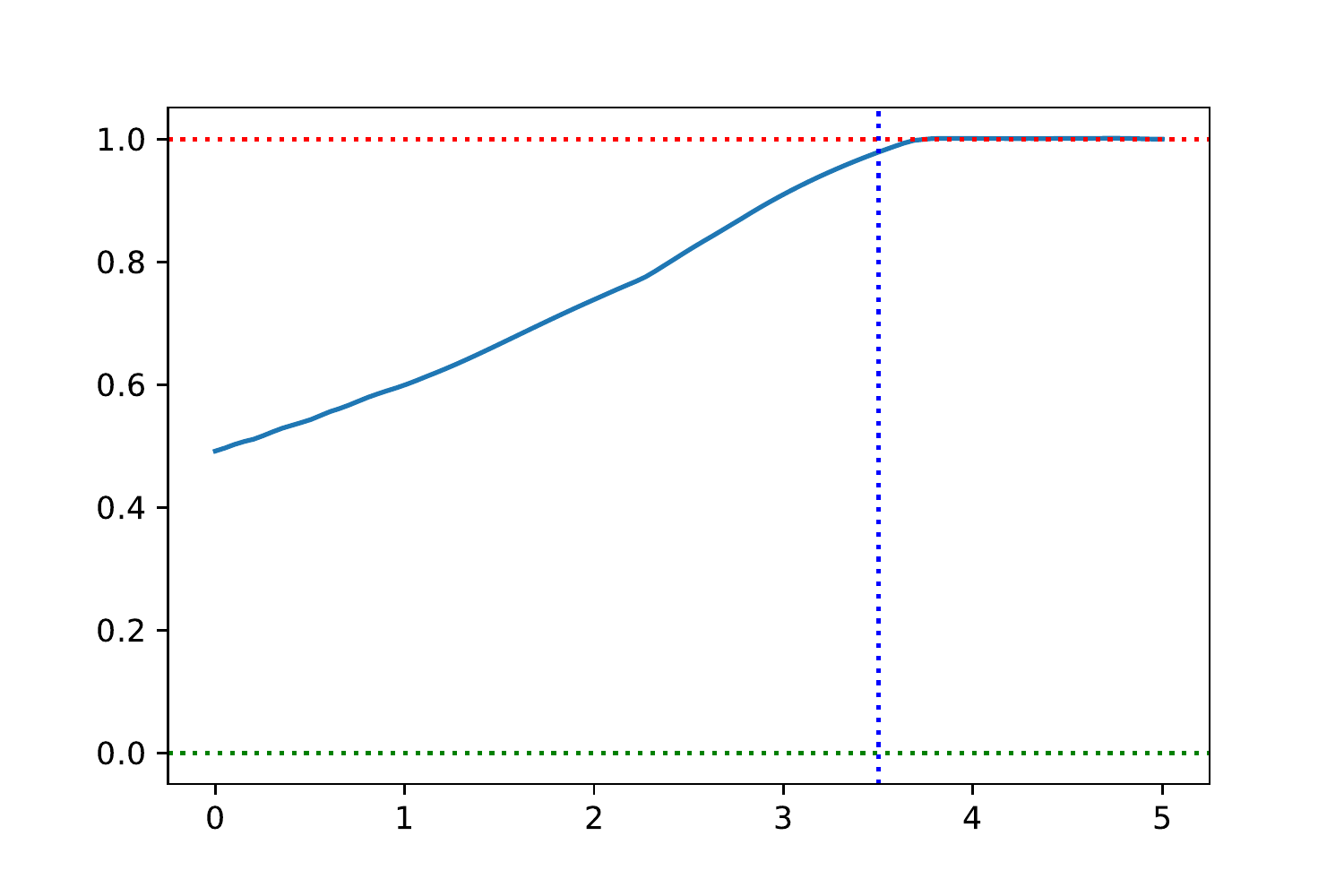}
    \includegraphics[width=70mm]{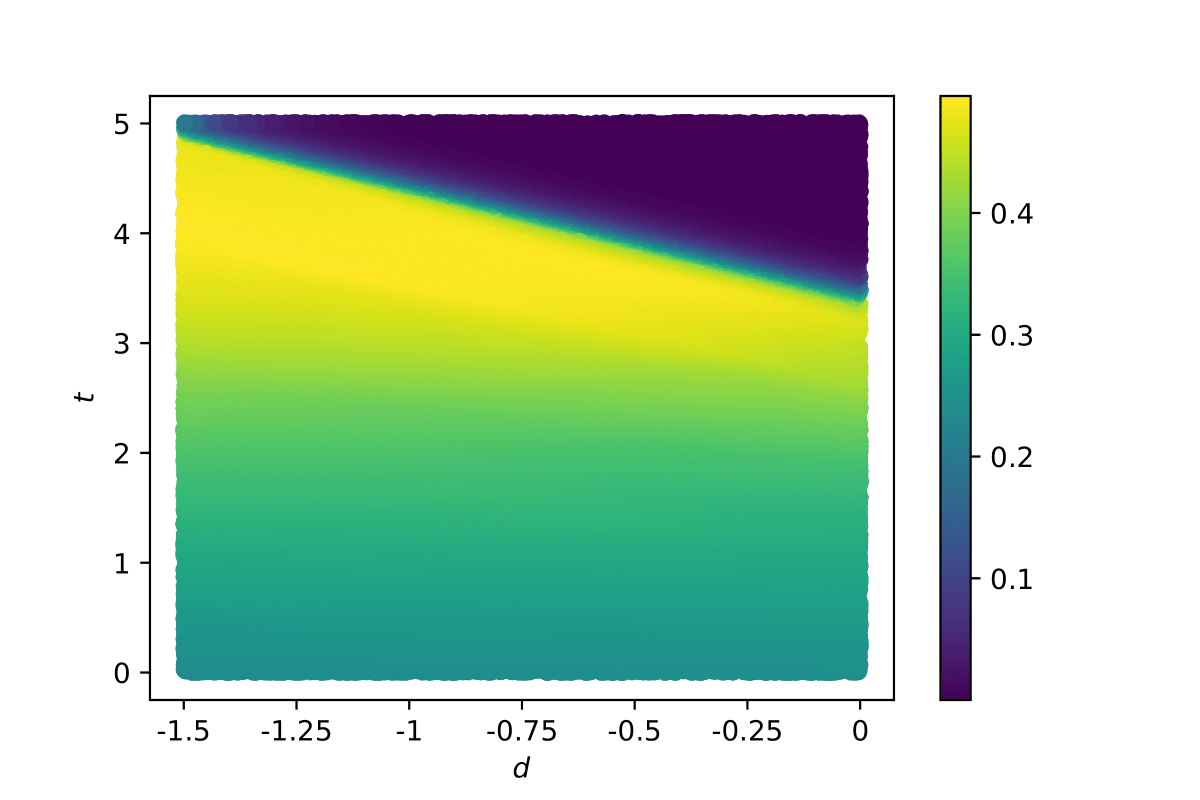}
    }
\vspace{-0.3cm}
\subfloat[Right: $t \mapsto P_{\hat{22}}(t,s)$. Left: $t,s \mapsto P_{22}(t,s,0)$.]{
\label{fig:horizon_T_2}
    \includegraphics[width=70mm]{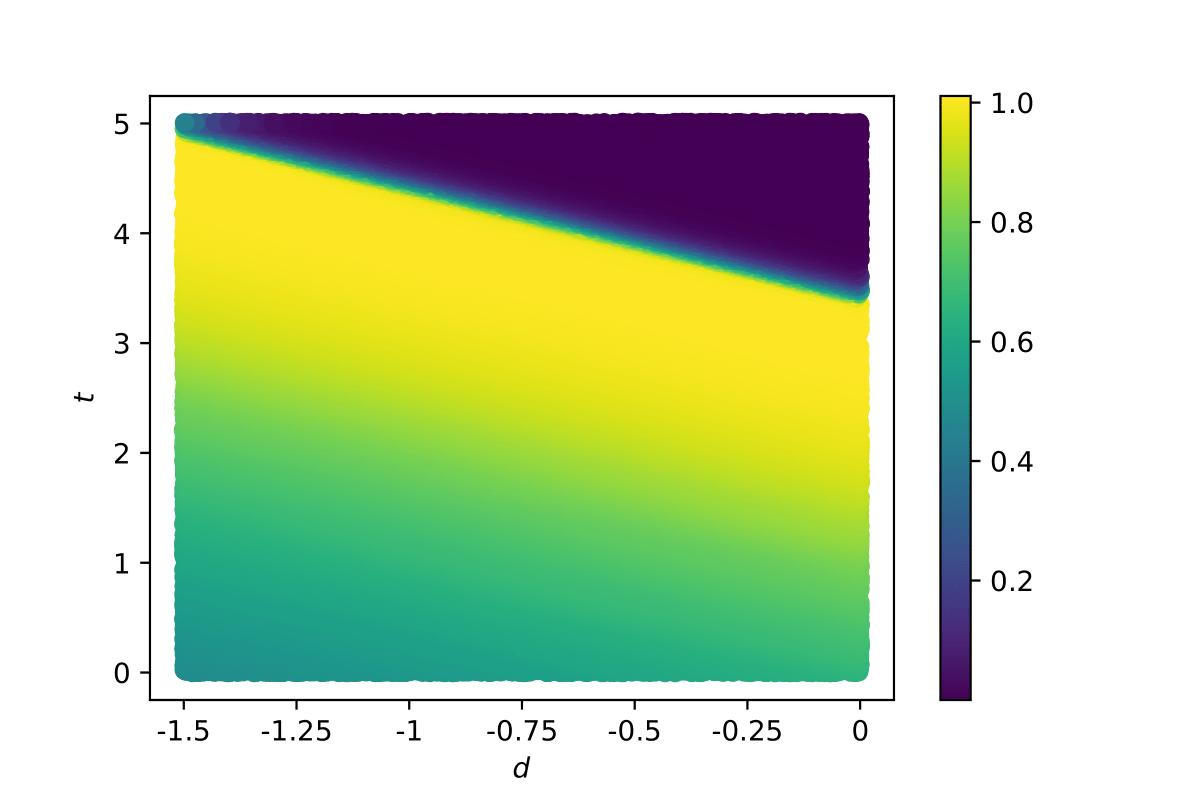}  
    \includegraphics[width=70mm]{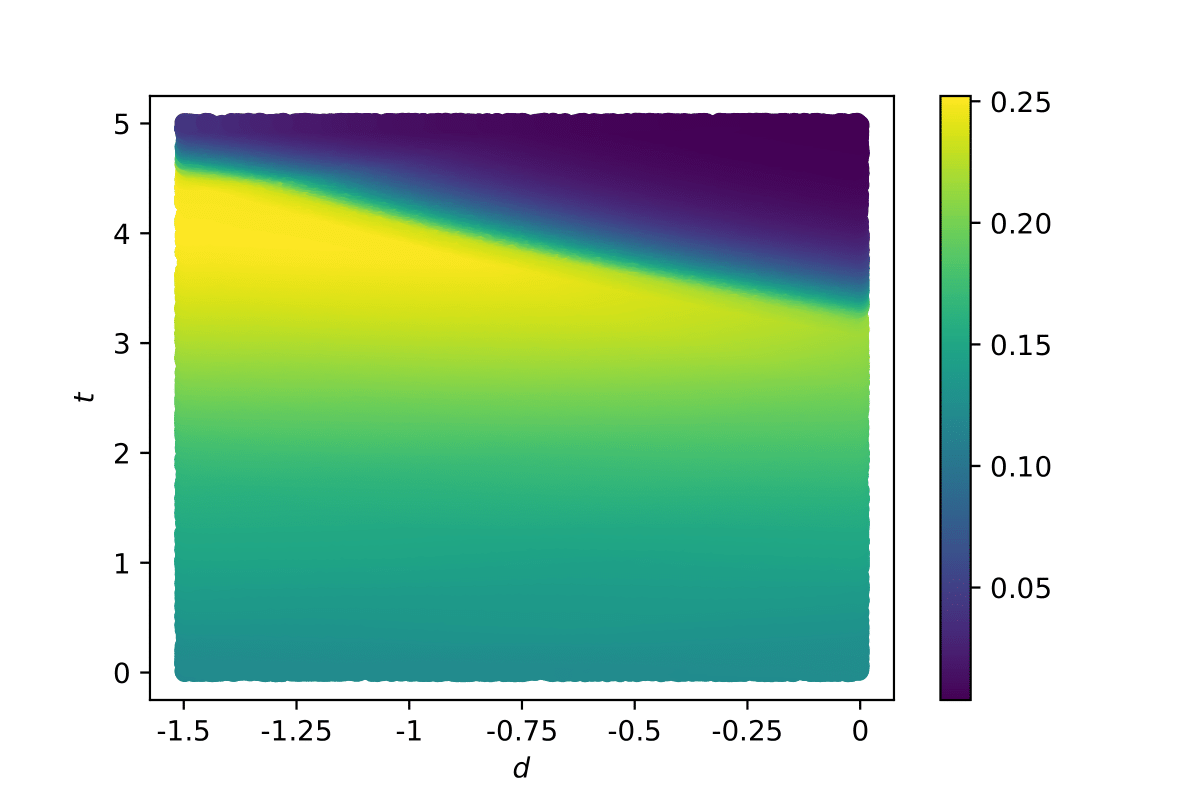}
    }
\caption{Numerical results of Algorithm \ref{algo:nn_training} with $\sigma=1$, $\lambda=0.5$, $d=1.5$, and $T=5$. } 
\label{eq:kernels}
\end{figure}

\subsection{One asset with delay and one without}

To further explore the effect of the delay on the control, we now study a toy example where the investor has two investment opportunities, one with a delayed execution and one without.  More precisely, consider the following portfolio dynamic
\bec{
    & dX_{t}^{\left(\alpha, \beta\right)}=\alpha_{t} \left\{ (\sigma_{1} \lambda_1) dt + \sigma_1 dW^1_t \right\} + \beta_{t-d} \left\{(\sigma_{2}\lambda_2) dt + \sigma_2 dW^2_t\right\},  \quad t\in [0,T], \\
    & X_0 = x_0, \quad \beta_s = \gamma_s, \quad  s\in [-d,0], \\
    & \langle W^1, W^2 \rangle_t = \rho t,
}
where $x_0 > 0$ and $\gamma \in L^2([-d, 0], \R)$, together with the same optimization objective \eqref{optimization_problem_mv} as before. Here, $\alpha_t$ and $\beta_t$ correspond respectively to the amounts of money the investor decides to invest at time $t$ in the undelayed and the delayed risky assets. The constants $ \lambda_i$ and $ \sigma_i$ represent respectively the risk premium and the volatility of the risky asset $i$. Following the heuristic approach of Section \ref{section:heuri}, we define the following set of Riccati-PDEs on  $[0,T]\times [-d,0]^2$
\begin{align}
    \label{eq:mutli_1}
	     &\dot{P}_{11}(t)  = \lambda_1^2 P_{11}(t) + \frac{P_{12}(t,0)^2}{P_{\hat{22}}(t,0)},\\
	     &(\partial_t-\partial_s)(P_{12})(t,s) = \lambda_1^2 P_{12}(t,s) + \frac{P_{12}(t,0)P_{22}(t,s, 0)}{P_{\hat{22}}(t,0)},\\
	     &(\partial_t - \partial_s)(P_{\hat{22}})(t,s)=0, \\
	     &(\partial_t - \partial_s-\partial_r)(P_{22})(t,s,r) = \lambda_1^2 \frac{P_{12}(t,s)P_{12}(t,r)}{P_{11}(t)} + \frac{P_{22}(t,s,0)P_{22}(t,0,r)}{P_{\hat{22}}(t,0)},
\end{align}
accompanied by the boundary conditions, for almost any $t,s \in [0,T]\times [-d,0]$
\begin{align}
\label{eq:mutli_2}
	      &P_{12}(t,-d) = \lambda_2\sigma_2\left( 1 - \rho\frac{\lambda_1}{\lambda_2} \right) P_{11}(t), &&
	      P_{\hat{22}}(t,-d) = \sigma_2^2 \left( 1-\rho^2 \right) P_{11}(t), \\
	      &P_{22}(t,s, -d)= \lambda_2\sigma_2\left( 1 - \rho\frac{\lambda_1}{\lambda_2} \right) P_{12}(t,s), && P_{22}(t,-d, s)= \lambda_2\sigma_2\left( 1 - \rho\frac{\lambda_1}{\lambda_2} \right) P_{12}(t,s),
\end{align}
and the terminal constraints 
\begin{align}
    \label{eq:mutli_3}
	      &P_{11}(T) = 1, &&
	        P_{12}(T,s) = P_{\hat{22}}(T,s) = P_{22}(T,s,r)=0,
\end{align}
for almost every $s,r \in [-d,0]$.  \\
As in the previous section, we first solve the inner optimization problem \ref{eq:inner_mv}.

\begin{lemma}
Fix $\eta\in \R$ and $\xi=c-\eta$. Assume \eqref{eq:mutli_1}-\eqref{eq:mutli_2}-\eqref{eq:mutli_3} admits a piecewise absolutely continuous solution with $P_{\hat{22}}(t)>0$  for any $t \leq T-d$, and define the couple $\left(\alpha^*, \beta^* \right) (\xi)$ as the investment strategies
\begin{align}
    &\alpha_{t}^{*}(\xi)= - \left\{ \frac{\lambda_{1}}{\sigma_{1}}\left(  X_{t}^* - \xi\right) + \rho\frac{\sigma_{2}}{\sigma_{1}} \beta_{t-d}^{*} + \frac{\lambda_{1} }{\sigma_{1}P_{11}(t)}\int_{t-d}^{t}\beta_{s}^{*}(\xi)P_{12}(t,s-t)ds \right\}.
    \\
    &\beta_{t}^{*}(\xi)= \frac{-1_{t \leq T-d}}{P_{\hat{22}}(t,0)}\left\{ P_{12}(t,0) \left(  X_{t}^* - \xi\right) +\int_{t-d}^{t}\beta_{s}^{*}(\xi) P_{22}(t,0,r-t)dr\right\},
\end{align}
where $X^*$ denotes the state process $X^{\left(\c^*, \beta^*\right)}$. Then, the inner minimization problem \eqref{eq:inner_mv} admits $(\alpha^*(\xi), \beta^*(\xi))$ as an optimal feedback strategy and the optimal value is 
\bes{
    V_0(\xi) =& P_{11}(0) (x_0-\xi)^2  + R({x}_0-\xi, \gamma),\\
}
where $R(\gamma)$ denotes the cost associated to the initial investment strategy $\gamma$ on $[-d, 0]$
\bes{
    R(x, \gamma) =&  2  x \int_{-d}^0 \gamma_s P_{12}(0,s) ds + \int_{-d}^0 \gamma^2_s P_{\hat{22}}(0,s) ds + \int_{[-d,0]^2} \gamma_s\gamma_u P_{22}(0,s, r)  ds dr.
}
\end{lemma}
\begin{proof}
  The proof is similar to the one of Lemma \ref{L:verification_mv}.
\end{proof}
Finally, the parameter $\eta^*$ and efficient frontier $\text{Var}(X_T^*) = f(c)$ are given by the same formulas \eqref{eq:eta_star} and \eqref{eq:value_final} as in in the mono-asset case, $\gamma$ being the pre-investment strategy of the delayed asset.

\begin{remark}
    One surprise that emerges is that the "\textit{buy the good stock sell the bad one}" criterion is unchanged for the delayed asset. Indeed, the sign of the control for this asset is still given by the sign of $1-\rho \frac{\lambda_1}{\lambda_2}$, that fixes the sign of the $P_{12}$ and $P_{22}$, as it would be in the case without delay\footnote{When $d=0$, recall that $\c_t^* = \frac{\lambda_1P_t}{\sigma_1(1-\rho^2)}(1-\rho \frac{\lambda_2}{{\lambda_1}})(\xi^*-X_t^*)$ and $\beta^* = \frac{\lambda_2P_t}{\sigma_2(1-\rho^2)}(1-\rho \frac{\lambda_1}{{\lambda_2}})(\xi^*-X_t^*)$ with $P$ being a positive function and $\xi^* \geq X^*$. Thus, in the classical setting, the buy or sell thresholds are $(1-\rho \frac{\lambda_2}{{\lambda_1}})$ and $(1-\rho \frac{\lambda_1}{{\lambda_2}})$.}, see the boundary conditions \eqref{eq:mutli_2}. But this threshold disappears in the undelayed asset's control as now only the term $\frac{\lambda_1}{\sigma_1}$ remains in the mean-reverting term. 
\end{remark}

\textbf{Numerical simulations: }
To exhibit the effect of the correlation $\rho$, we generate two independent Brownian motions $\left( W_t^1 \right)_{t\in[0,T]}$ and $\left( B_t \right)_{t\in[0,T]}$ and define the Brownian motion $\left( W_t^2 \right)_{t\in[0,T]}$ as
\begin{align*}
    W_t^2 = \rho W_t^1 + \sqrt{1-\rho^2}B_t,\quad t\in [0,T].
\end{align*}
We then compare different scenarios with different values of correlation $\rho$ and delay $d$ while fixing $W^1$ and $B$. The numerical simulations can be found in Figures \ref{fig:matrix}, \ref{fig:rho} and \ref{fig:d}.
As it could have been expected, we see from \eqref{eq:mutli_2} and Figure \ref{fig:matrix}, that the greater $\rho$ is, the more favored the undelayed asset is.

\newpage
\null
\vfill

\begin{figure}[h!]
\begin{tabular}{cccc}
 & $d=1.5$ & $d=1$ & $d=0.5$ \\
  \rotatebox{90}{ $\;\;  \rho=-0.2$} & 
  \includegraphics[width=0.3\linewidth]{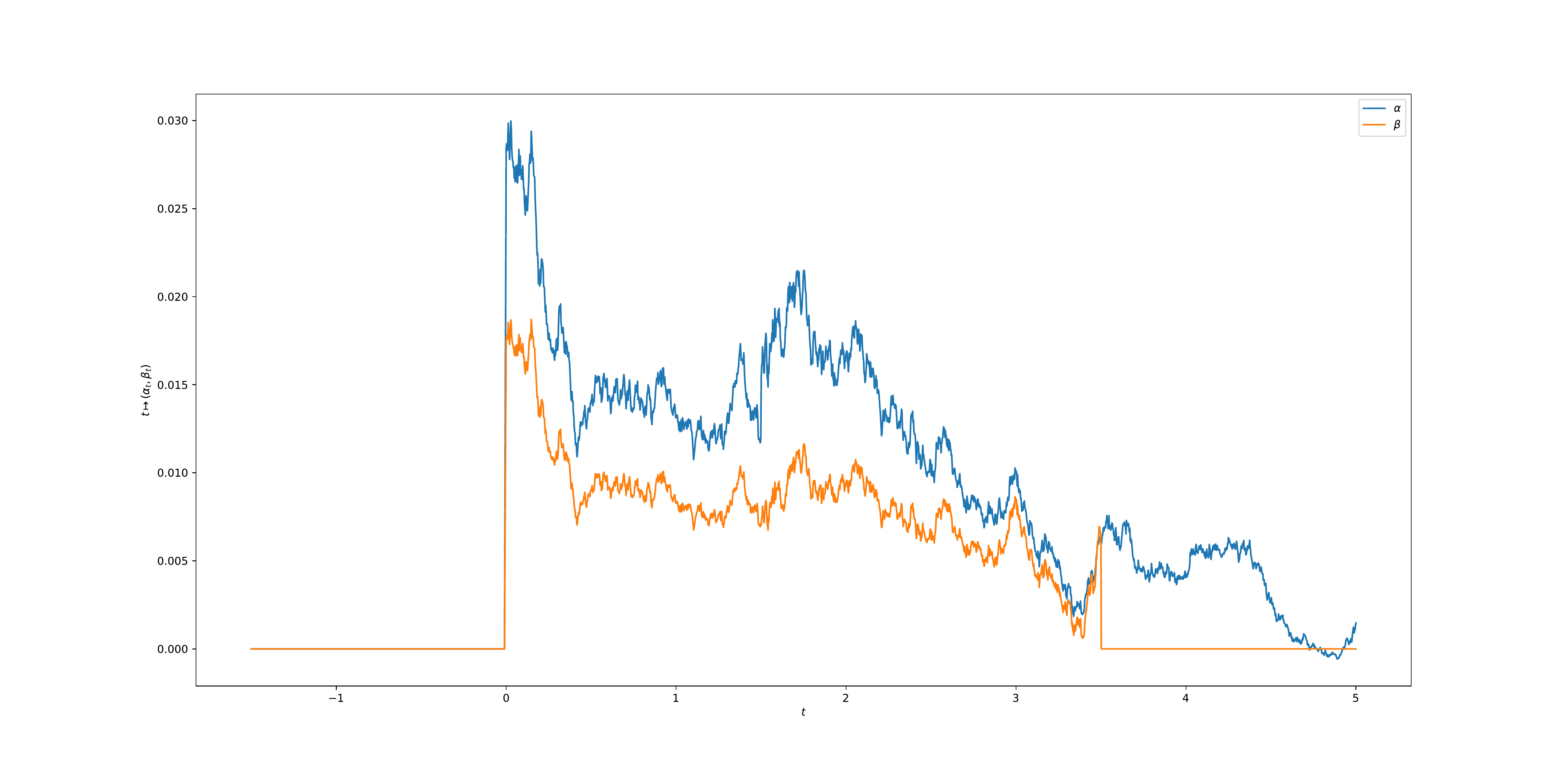} &   \includegraphics[width=0.3\linewidth]{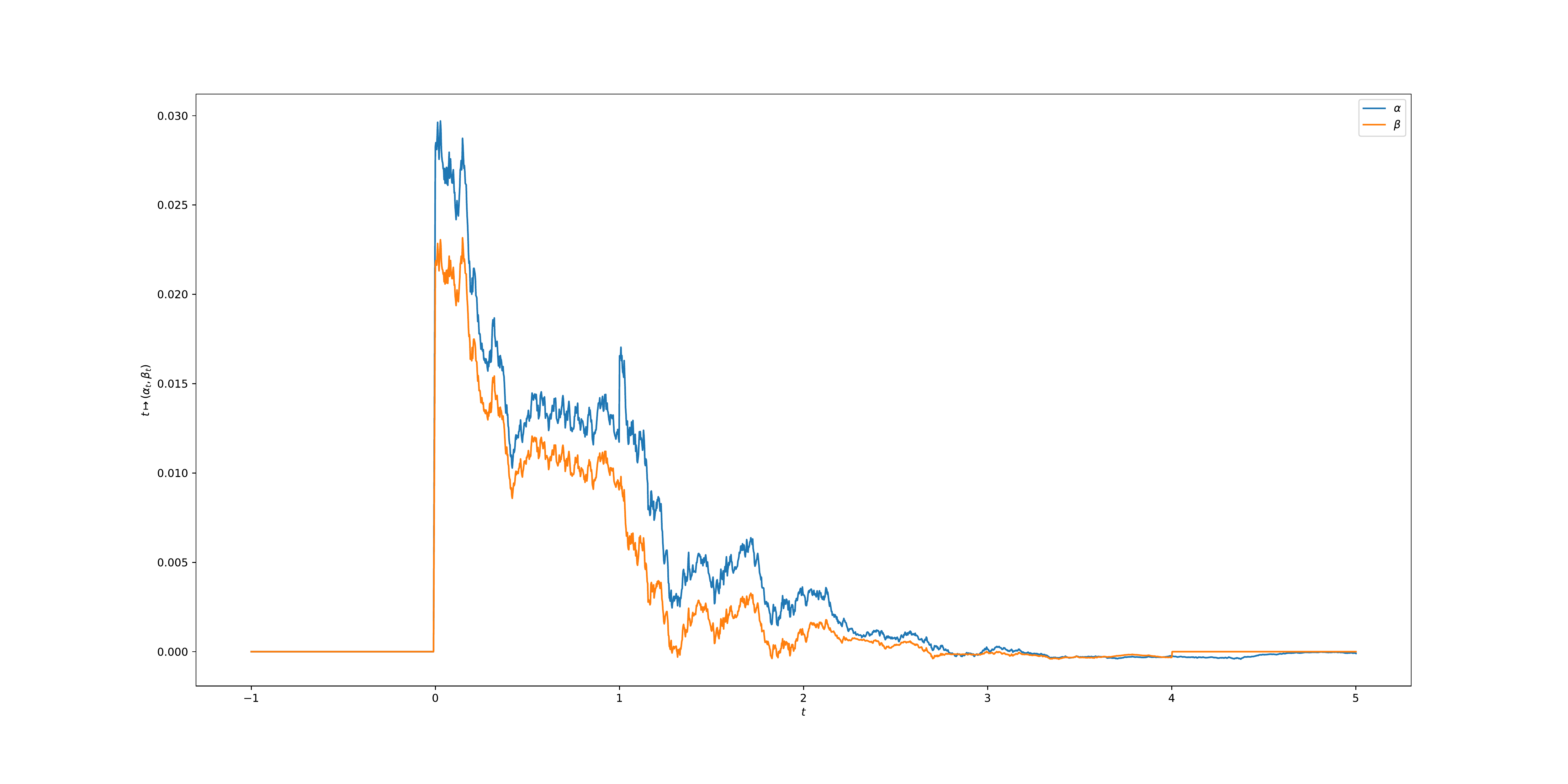} &
  \includegraphics[width=0.3\linewidth]{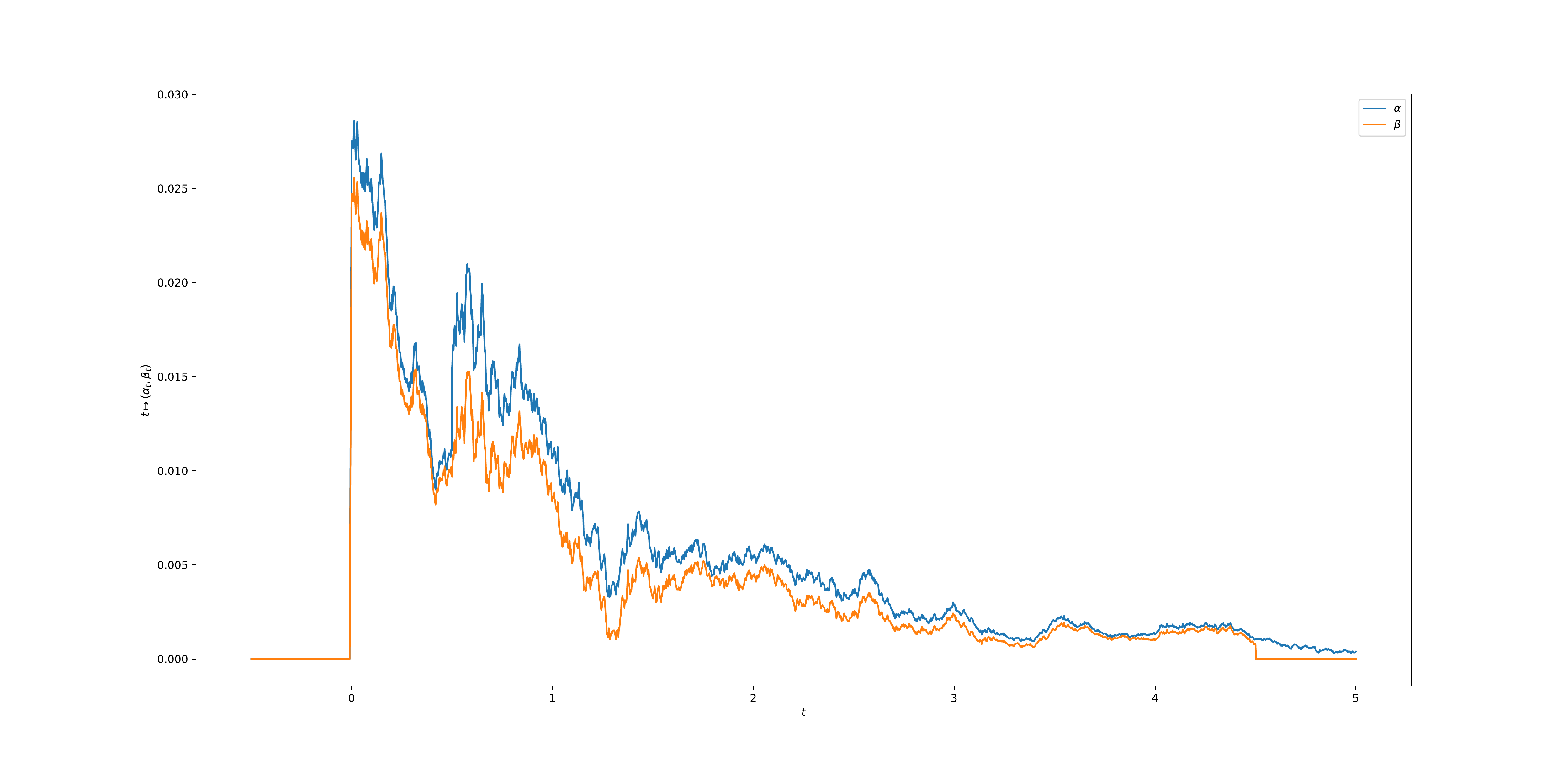}\\
 \rotatebox{90}{ $\;\;\;\; \; \rho=0$} &
 \includegraphics[width=0.3\linewidth]{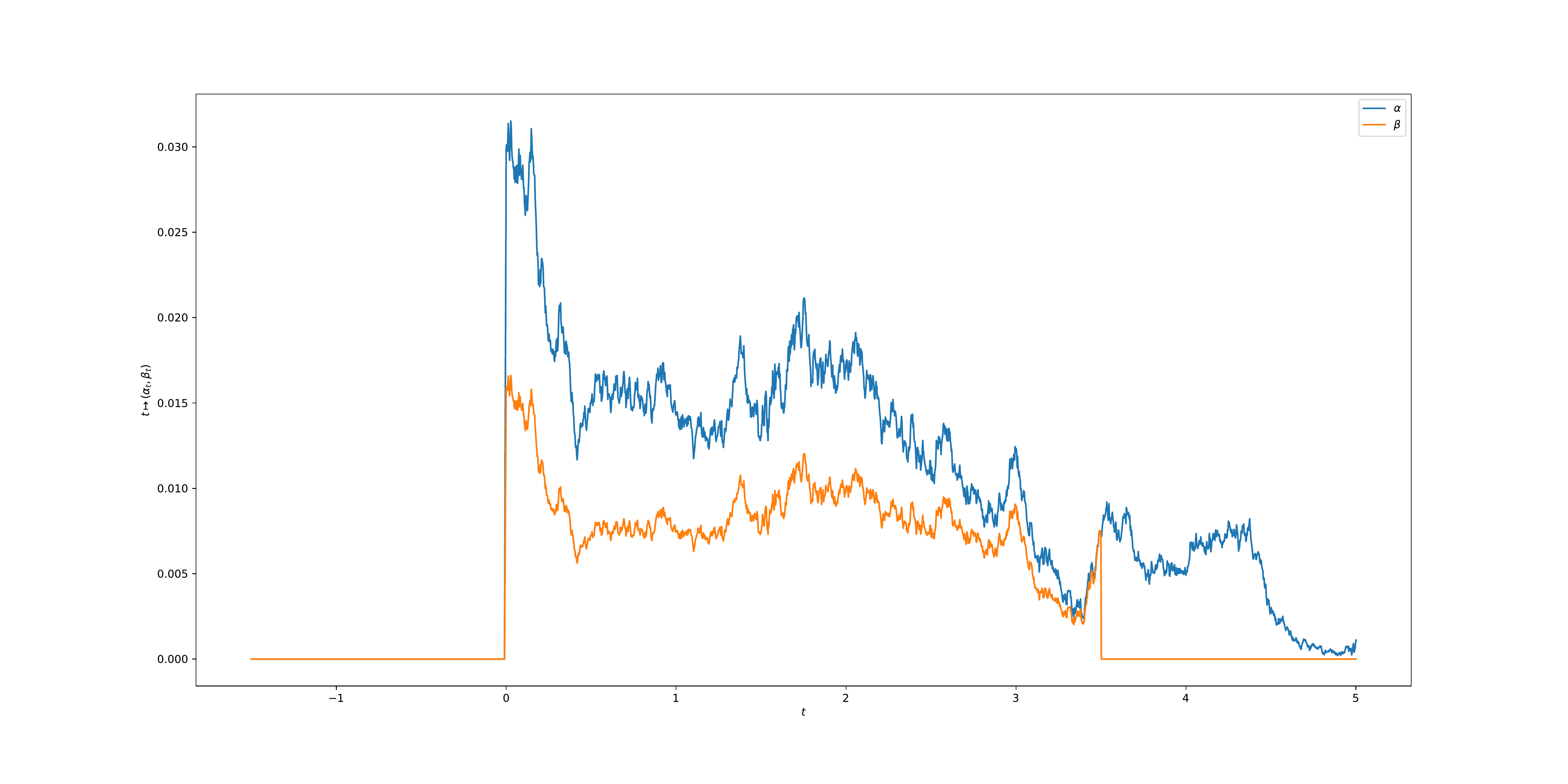} &   \includegraphics[width=0.3\linewidth]{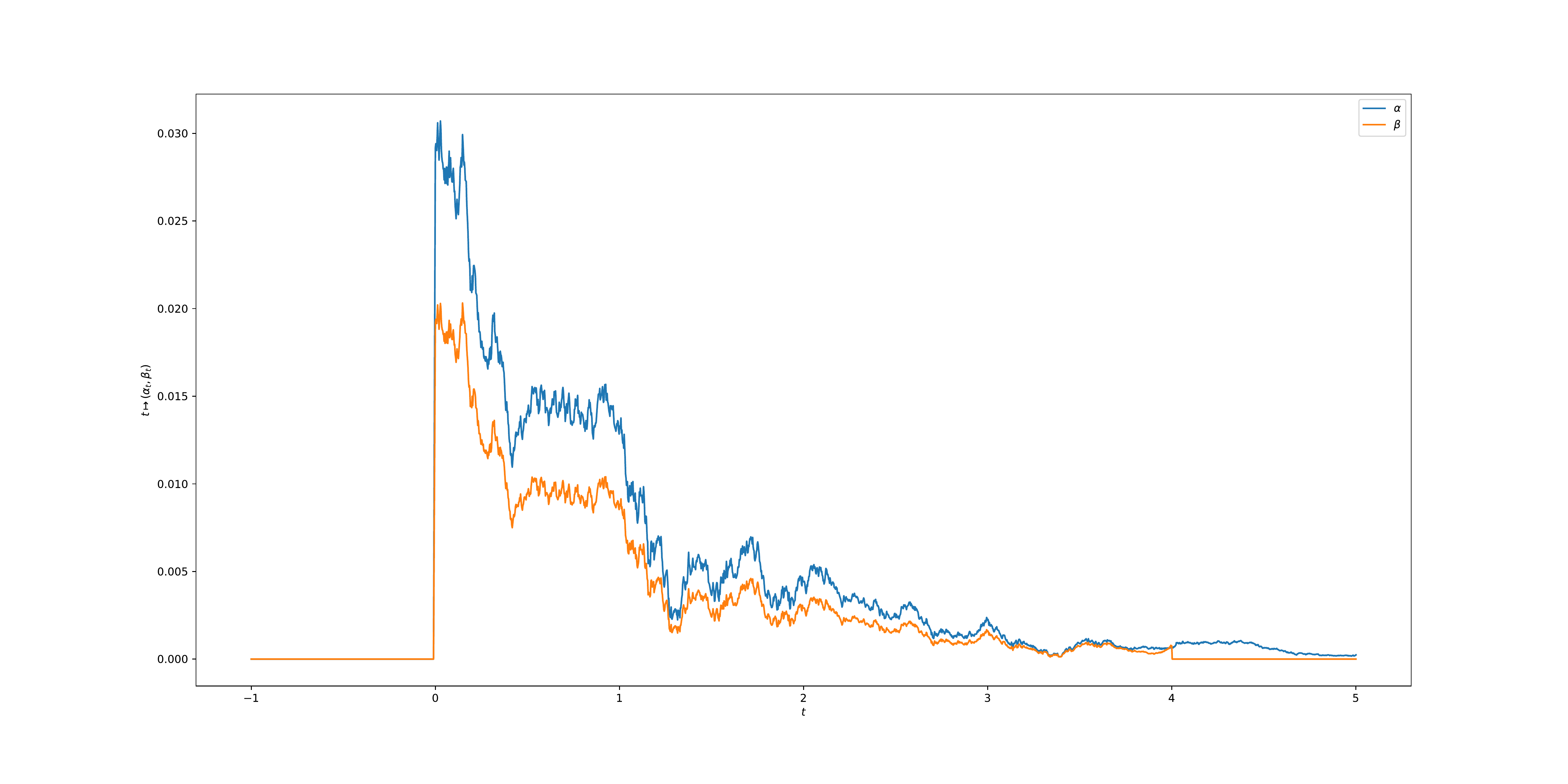} &
 \includegraphics[width=0.3\linewidth]{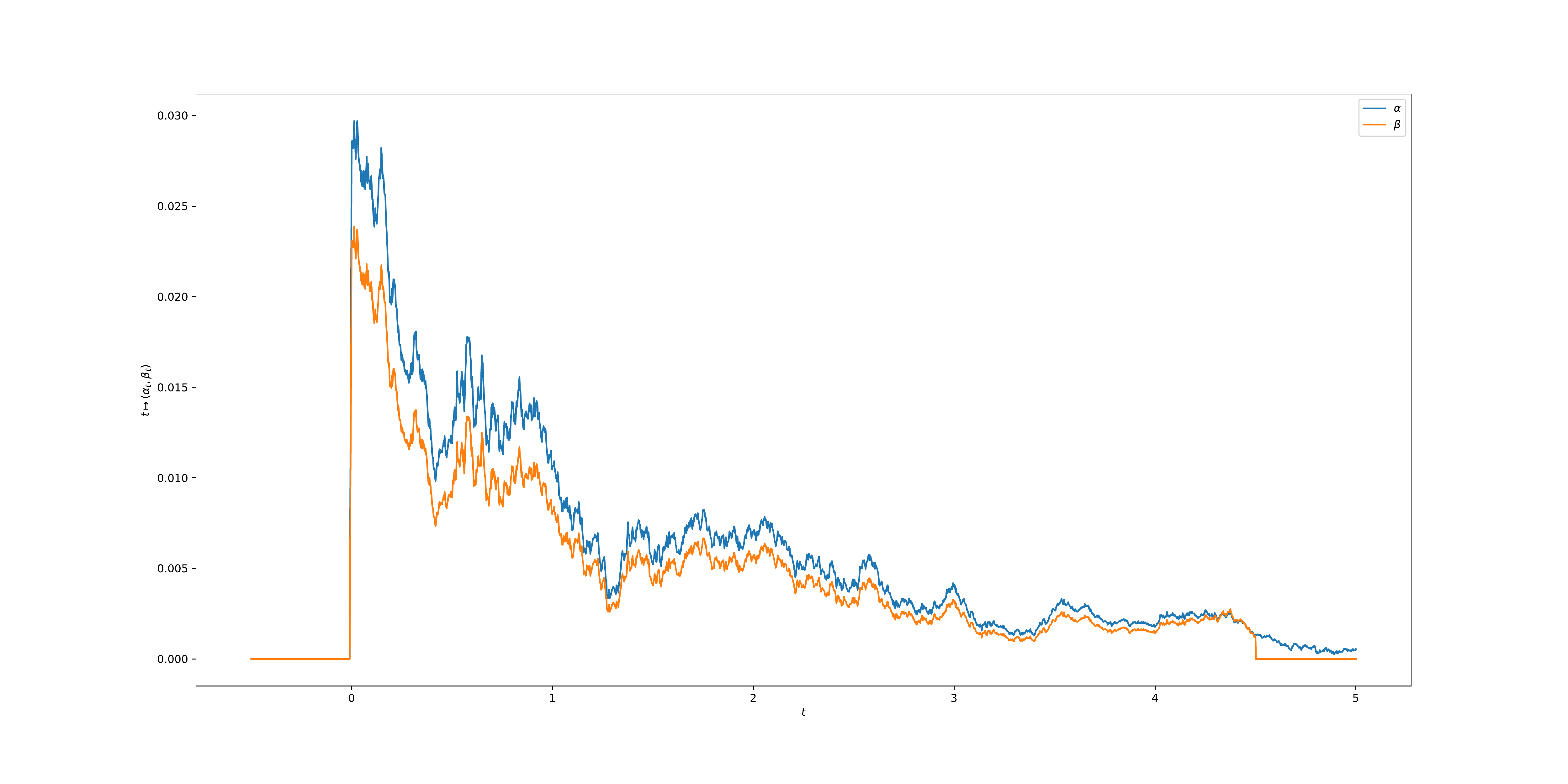}\\
 \rotatebox{90}{ $\;\; \;\; \rho=0.7$} &
 \includegraphics[width=0.3\linewidth]{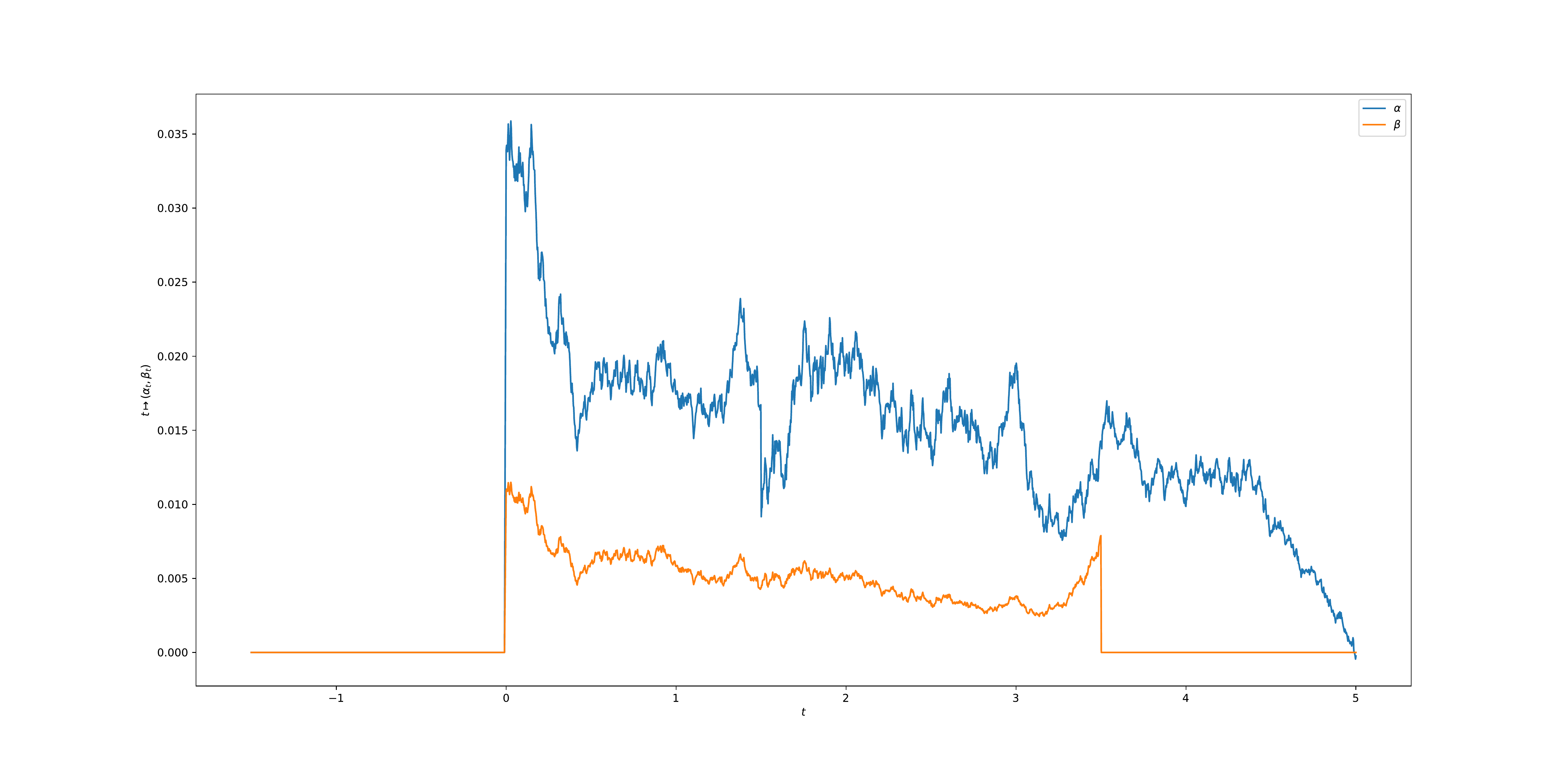} &   \includegraphics[width=0.3\linewidth]{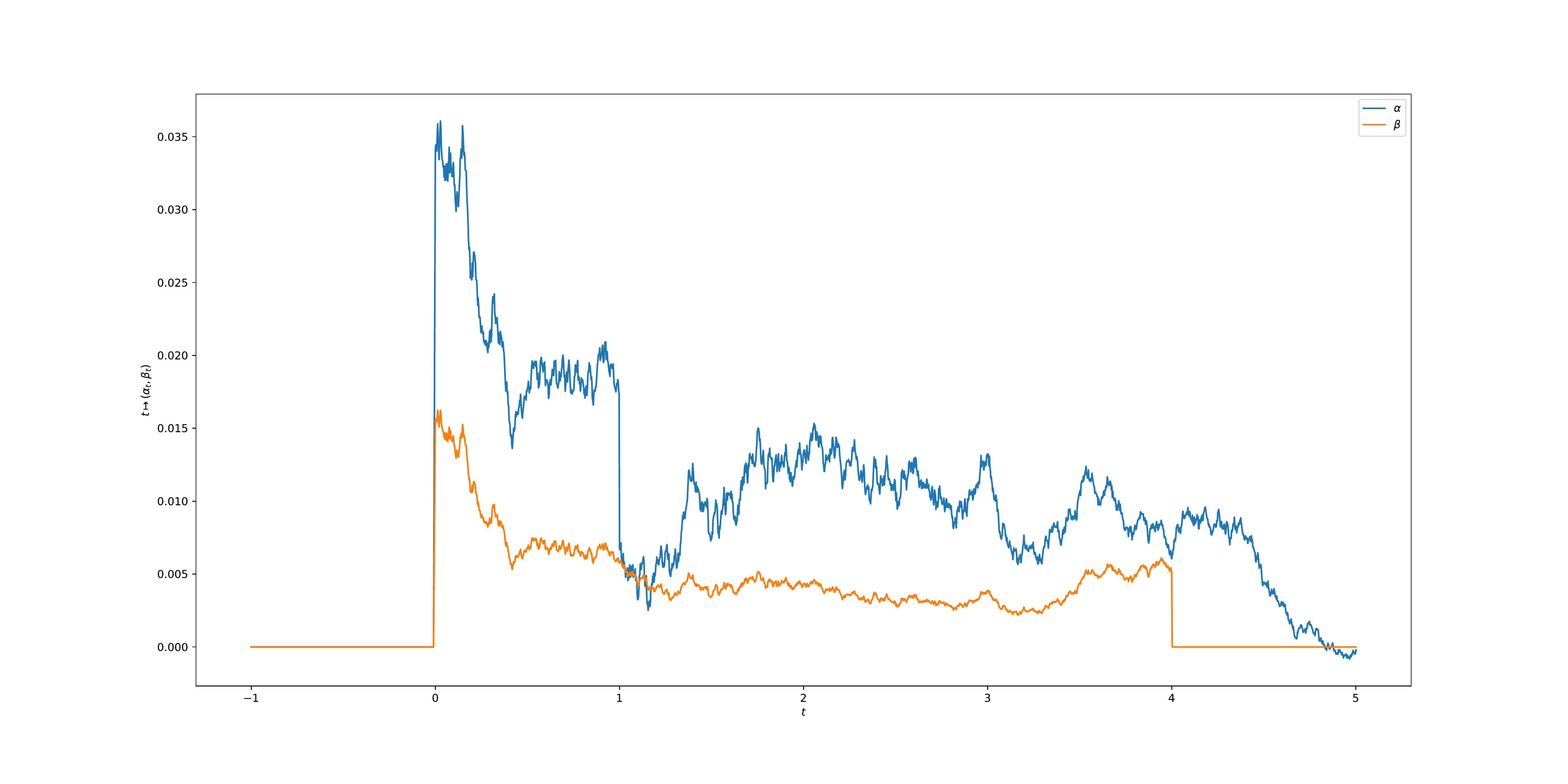} &
 \includegraphics[width=0.3\linewidth]{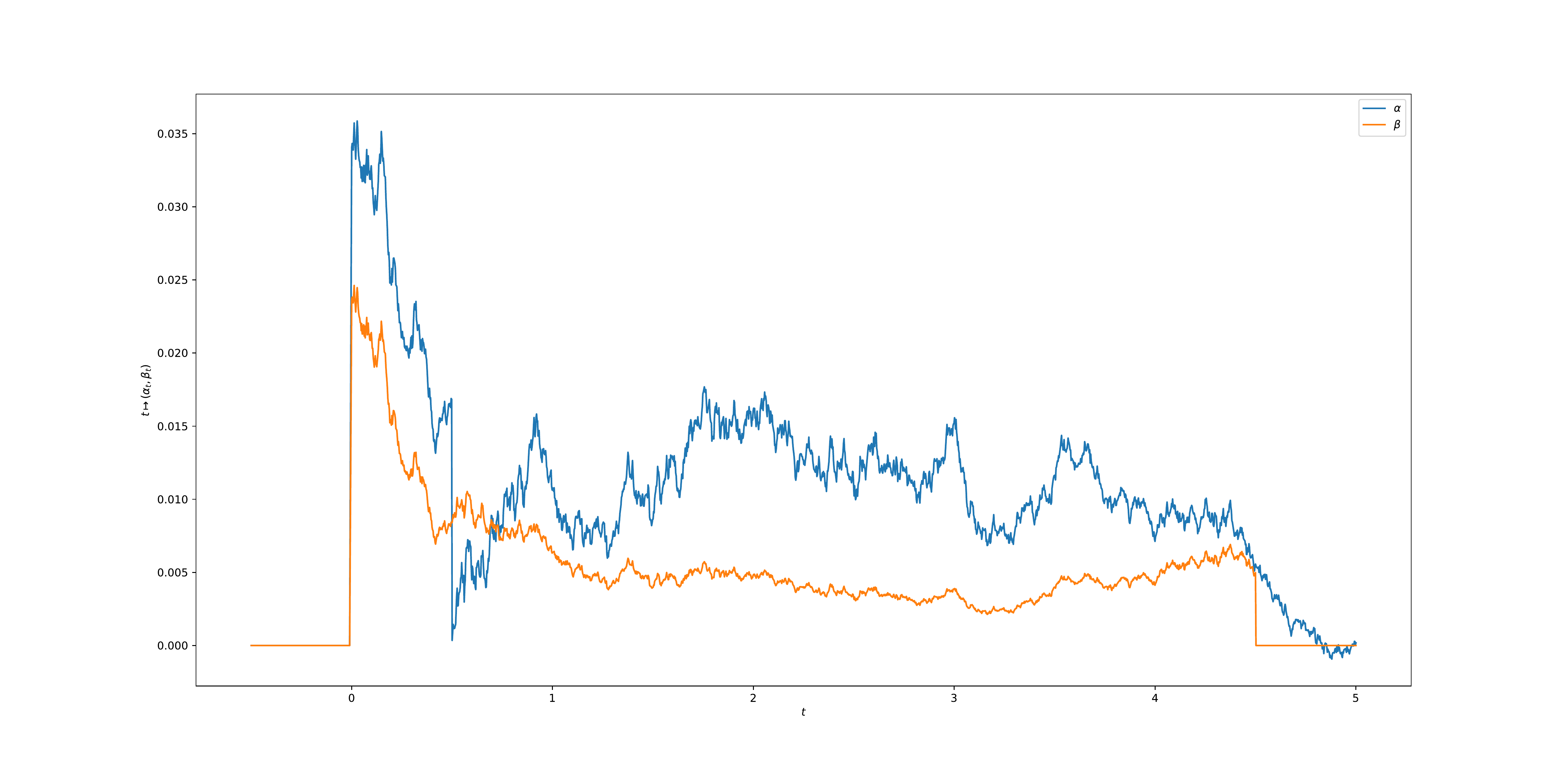}\\
\end{tabular}
\caption{$t \mapsto \left(\c_t^*,\beta_t^*\right)$, with $\sigma_1=\sigma_2=1$, $\lambda_1=\lambda_2=0.5$ and $T=5$. Blue : $\c^*$, orange : $\beta^*$. The same realizations of $W$ and $B$ were used for all experiments. Note that the more positively correlated the assets are, the more favored the undelayed asset is.}
\label{fig:matrix}
\end{figure}
\vfill

\newpage

\begin{figure}[h!]
\centering
{
    \includegraphics[width=0.97\linewidth]{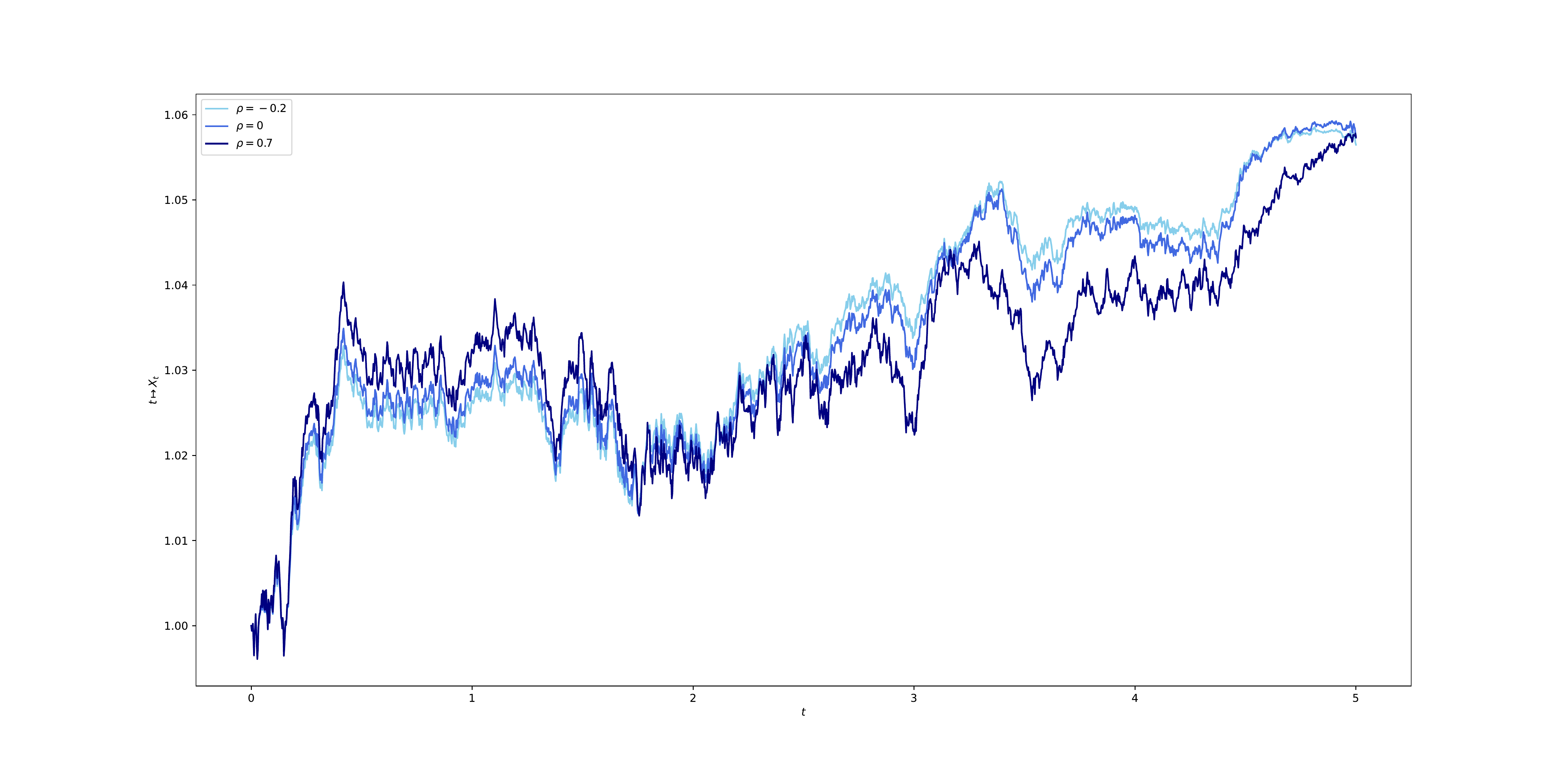}
    }
\caption{$t \mapsto X_t^*$, with $\sigma_1=\sigma_2=1$, $\lambda_1=\lambda_2=0.5$, $T=5$ and $d=1.5$ for $\rho=-0.7$, $0$ and $0.7$. The same realizations of the Brownian motions $W^1$ and $B$ was used for all experiments.}
\label{fig:rho}
\end{figure}


\begin{figure}[h!]
\centering
{
    \includegraphics[width=0.97\linewidth]{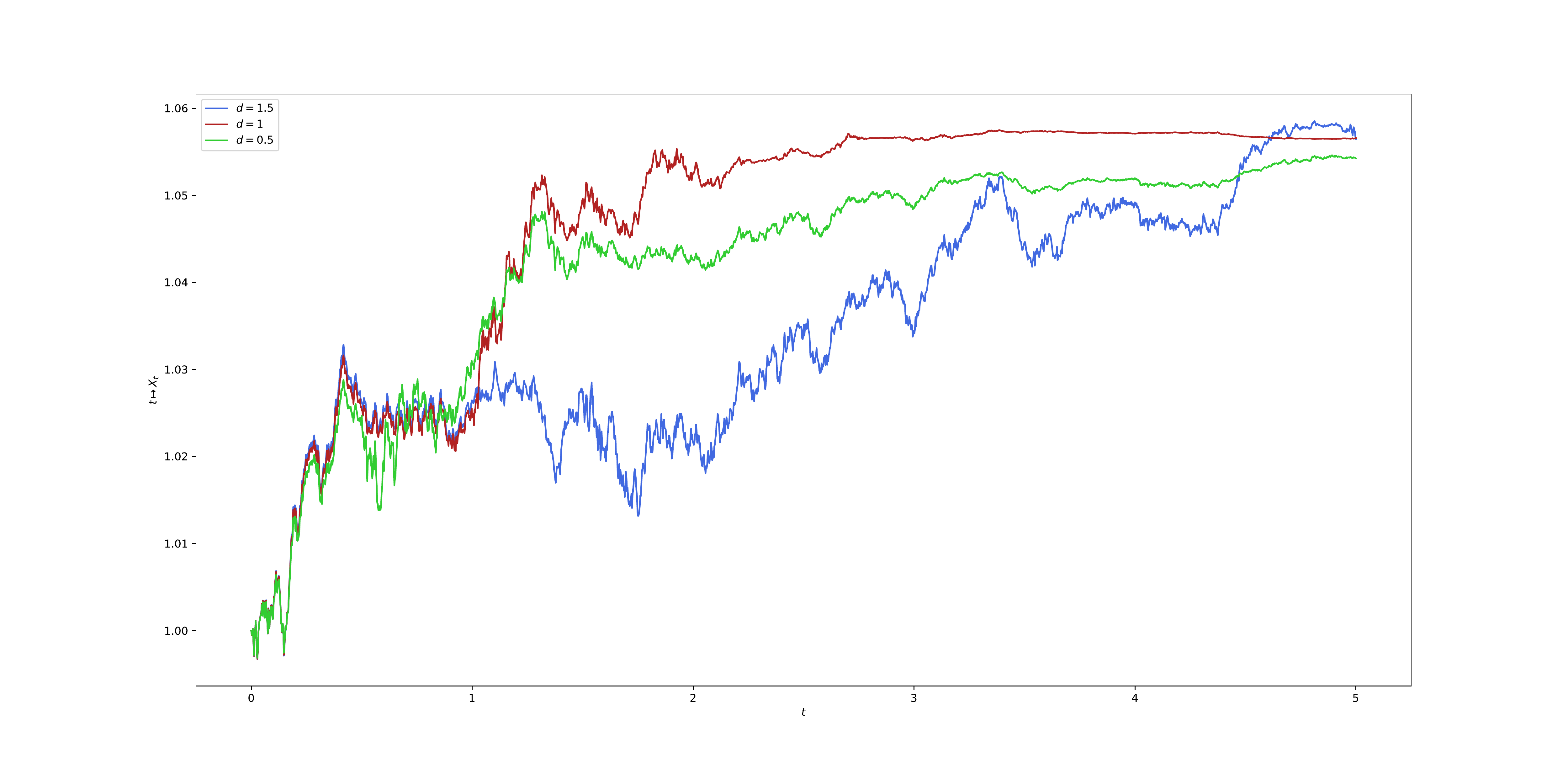}
    }
\caption{$t \mapsto X_t^*$, with $\sigma_1=\sigma_2=1$, $\lambda_1=\lambda_2=0.5$, $T=5$ and $\rho=-0.7$ for $d=1.5$, $1$ and $0.5$. The same realizations of the Brownian motions $W^1$ and $B$ was used for all experiments. }
\label{fig:d}
\end{figure}

\newpage

\bibliographystyle{plainnat}
\bibliography{bibl}

\appendix

\section{Proof of Proposition \ref{prop:ricatti}}
\label{appendix:fixed_point_banach}
Our proof extends \citet[see Theorem 5]{alekal1971quadratic} to the case where the volatility is controlled. It consists in slicing the domain $\Dcal$ in slices of size $d$ and proceeding by a backward recursion. More precisely, we show existence and uniqueness over a sequence of slices $\left([T-(n+1)d, T-nd] \times [-d, 0]^2\right)_{n}$. We then concatenate the sequence of absolutely continuous solutions obtained, which yields a piece-wise absolutely continuous solution. In each slice, the proof consists of the following steps 
\begin{enumerate}[label=({\arabic*})]
    \item Show that there exists a unique solution on a small interval;
    \item Prove that the local solution is Lipschitz;
    \item As a result extend the solution to the whole slice.
\end{enumerate}
We finally concatenate the sequence of solutions obtained above.

\subsection{Slice $t \in [T-d, T]$, initialization}
On $\Dcal_b \cup \Dcal_c$, the constraints \eqref{eq:a}-\eqref{eq:b}-\eqref{eq:c} on $P_{12}, P_{\hat{22}}$ and $P_{22}$ reduce to linear homogeneous transport equations admitting closed form solutions given, for every $(t,s,r)\in \mathcal{D}_{b} \cup \mathcal{D}_c$, by
\begin{align}
     &P_{12}(t,s)  = b P_{11}(t+s+d)1_{t+s+d \leq T}, && P_{\hat{22}}(t,s)  = \sigma^2 P_{11}(t+s+d)1_{t+s+d \leq T},  \\
      & P_{22}(t,s,r)  = b^2 P_{11}(t+ s \vee r + d)1_{t+s\vee r+d \leq T}. 
\end{align}
Or, as $P_{11}(t) = 1$ for any $t \geq T-d$, we then have for every $(t,s,r)\in \mathcal{D}_{b} \cup \Dcal_c$ 
\begin{align}
     &P_{11}(t) = 1, && P_{12}(t,s)  = b1_{t+s+d \leq T},    \\
      &P_{\hat{22}}(t,s)  = \sigma^21_{t+s+d \leq T}, && P_{22}(t,s,r)  = b^21_{t+s\vee r+d \leq T}. 
\end{align}
The existence and uniqueness in the sense of Definition \ref{def:sol_E_i_simplified} are thus trivially proved on $[T-d, T]$.

\subsection{Slice $[T - 2d, T- d]$}

On $[T-2d, T-d]\times [-d, 0]^2$, we have $P_{\hat{22}}(t,s) = \sigma^2 P_{11}(t+s+d)$ so that $P_{\hat{22}}(t,0)  = \sigma^2 P_{11}(t+d)=\sigma^2$. Consequently, the system \eqref{eq:a}-\eqref{eq:b}-\eqref{eq:c} reduces to 
\begin{align}
\label{eq:E_i_system_first_slice}
    &\dot{P}_{11}(t) =\frac{P_{12}(t,0)^{2}}{\sigma^2}, \\
    &(\partial_t - \partial_s)(P_{12})(t,s)=\frac{P_{12}(t,0)P_{22}(t,s,0)}{\sigma^2}, \\
    & (\partial_t - \partial_s-\partial_r)(P_{22})(t,s,r)  =\frac{P_{22}(t,s,0)P_{22}(t,0,r)}{\sigma^2},
\end{align}
with terminal conditions 
\begin{align}
\label{eq:Ei_terminal_first_slice}
    & P_{11}(T-d)  =1, && P_{12}(T-d, s) = b, && P_{22}(T-d, s, r) = b^2,
\end{align}
and boundary constraints
\begin{align}
\label{eq:Ei_constraints_first_slice}
    & P_{12}(t,-d)  = b P_{11}(t),  && P_{22}(t,s,-d)  = b P_{12}(t,s).
\end{align}

\noindent Thus, for every $(t,s,r)\in[T-2d, T-d]\times[-d,0]^2$, the set of equations \eqref{eq:E_i_system_first_slice} and constraints \eqref{eq:Ei_terminal_first_slice}-\eqref{eq:Ei_constraints_first_slice} can be rewritten in the following integral form 
\bes{
\label{eq:system_integral_form_first_slice}
    P_{11}(t) &= 1- \sigma^{-2}\int_{t}^{T-d} {P_{12}(x,0)^{2}}dx, 
    \\
    P_{12}(t,s) &= b P_{11}((T-d) \wedge (t+s+d)) \\
    &\;\;\;\;- \sigma^{-2} \int_{t}^{(T-d) \wedge (t+s+d)} {P_{12}(x,0)P_{22}(x,t+s-x,0)}dx, 
    \\
    P_{22}(t,s,r)  &= b P_{12}((T-d) \wedge (t+s \wedge r+d),(s-r) \vee (r-s)-d)\\
    &\;\;\;\; - \sigma^{-2} \int_{t}^{(T-d) \wedge (t+s \wedge r+d)} {P_{22}(x,t+s-x,0)P_{22}(x,0,t+r-x)}dx.
}
We then make use of the following lemma to prove local existence of a solution.
\begin{lemma}
\label{lemma:existence_first_slice}
There exists $\tau \in (0, d]$ such that system \eqref{eq:system_integral_form_first_slice} has a unique absolutely continuous solution on $[T-\tau-d, T-d]\times [-d, 0]^2$. 
\end{lemma}

\begin{proof}
Let $\tau \in (0, d]$ and $\Scal_\tau$ denote the Banach space of absolutely continuous functions $\xi=\left(\xi_{1}(\cdot),\xi_{2}(\cdot,\cdot),\xi_{3}(\cdot,\cdot,\cdot)\right)$
defined on 
$$\mathcal{D}_{\tau}=\left\{ (t,s,r)|\ T-d-\tau\leq t\leq T-d, -d\leq s,r\leq 0 \right\},$$
endowed with the sup-norm
\begin{equation}
    \|\xi\|_\infty=\|\xi_1\|_\infty + \|\xi_2\|_\infty + \|\xi_3\|_\infty,
\end{equation}
where $\|\xi_1\|_\infty, \|\xi_2\|_\infty$ and $\|\xi_3\|_\infty$ denote, with a slight abuse of notation, the respective sup-norm on $[T-d-\tau, T-d]$, $[T-d-\tau, T-d]\times[-d, 0]$ and  $[T-d-\tau, T-d]\times[-d, 0]^2$.
Let $\Bcal_\tau$ denote the ball in $\Scal_\tau$
\bes{
\Bcal_\tau = \{ (\xi_1, \xi_2, \xi_3) \in \Scal_\tau : \quad \|\xi_1-1\| \leq 1/2, \quad  \|\xi_2-b\| \leq |b|/2,  \quad  \|\xi_3 - b^2\| \leq b^2/2 \},
}
On $\Bcal_\tau$, we denote by $\phi=\left(\phi_{1},\phi_{2},\phi_{3}\right)$ the operator defined as follows 
\bes{
\left(\phi_{1}\xi\right)(t)  &= \; 1 - \sigma^{-2}\int_{t}^{T-d}{\xi_{2}(x,0)^{2}}dx\\
\left(\phi_{2}\xi\right)(t,s) &= \;  b \phi_1(\xi)((T-d) \wedge (t+s+d)) \\
&-\sigma^{-2}\int_{t}^{(T-d) \wedge (t+s+d)} {\xi_{2}(r,0)\xi_{3}(r,t+s-x,0)}dx\\
\left(\phi_{3}\xi\right)(t,s,r) &= \; {  b
\phi_2(\xi)\left((T-d) \wedge (t+s \wedge r+d),(s-r) \vee (r-s)-d \right)}\\
&\;\;\;\; -\sigma^{-2}\int_{t}^{(T-d) \wedge (t+s \wedge r+d)} {\xi_{3}(x,t+s-x,0)\xi_{3}(x,0,t+r-x)} dx.
}
Clearly, there exists $\tilde \tau>0$ such that for any $\tau  \leq \tilde \tau $, ${\red \phi}(\Bcal_{ \tau}) \to \Bcal_{ \tau}$.
We show a contraction property on $\phi$. For any $\xi, \xi' \in \Bcal_{ \tau}$, we have the following inequalities 
\bes{
\|\phi_{1}(\xi)-\phi_{1}(\xi')\|_\infty  \leq & \;\; 4\tau\sigma^{-2} |b| \| \xi_2 - \xi_2' \|_\infty,
\\
\|\phi_{2}(\xi)- \phi_{2}(\xi') \|_{\infty} \leq & 4\tau\sigma^{-2} \left(|b| \|\xi_2-\xi_2' \|_\infty + |b|^2 \| \xi_3-\xi_3' \|_\infty \right) 
\\
&+   |b| \|\phi_1(\xi) - \phi_1( \xi ')\|_\infty,
\\
\|\phi_{3}(\xi) - \phi_{3}(\xi')\|_\infty \leq & |b| \| \phi_2(\xi) - \phi_2(\xi')\|_\infty + 4\tau\sigma^{-2} |b|^2 \|\xi_3-\xi_3'\|_\infty.
 }
Consequently, the operator $\phi$ satisfies
\bes{
    \|\phi(\xi) - \phi(\xi')\|_\infty \leq \tau m \|\xi-\xi'\|_\infty,
}
where $m>0$ depends on $b$ and $\sigma$. Therefore, for $\tau < \tilde \tau \wedge m^{-1}$,
the operator $\phi$ is a contraction of $\Bcal_\tau$ into itself. Thus, $\phi$ admits a unique fixed point in $\Bcal_\tau$, which is solution to \eqref{eq:system_integral_form_first_slice} on $\Dcal_\tau$.\\
\end{proof}

\begin{lemma}
\label{lemma:lipschitz_first_slice}
Let $\xi = (\xi_1, \xi_2, \xi_3)$ denote  the absolutely continuous solution of \eqref{eq:system_integral_form_first_slice} on $\mathcal{D}_{\tau}$ from Lemma \ref{lemma:existence_first_slice}. Then $\xi$ is Lipschitz in each variable on $\Dcal_\tau$. 
\end{lemma}
\begin{proof}
As $\xi_1$, $\xi_2$ and $\xi_3$ are continuous on $\mathcal{D}_\tau$, there exists a constant $m>0$ such that $|\xi_1| \wedge |\xi_2| \wedge |\xi_3| \leq m$ on $\Dcal_\tau$. Thus, $\xi_1$ is Lipschitz with constant $\kappa=m^2\sigma^{-2}$. Let us now show that $\xi_2$ and $\xi_3$ are Lipschitz in the $s$-variable. Fix $t \in [T-d - \tau, T-d]$ and $\eta>0$. Then, for any $s \in [-d, 0]$ such that $s+\eta \in [-d, 0]$, we have
\bes{
    \left|\xi_2(t,s) - \xi_2(t,s+\eta) \right|\leq& \kappa\eta +{\sigma^{-2}} \Bigg| \int_{t}^{(T-d) \wedge (t+s+\eta+d)} \xi_2(x,0)\xi_3(x,t+s+\eta-x,0)dx\\
    & - \int_{t}^{(T-d) \wedge (t+s+d)} \xi_2(x,0)\xi_3(x,t+s-x,0)dx \Bigg| \\
    & \leq \kappa \eta+ \bold{I}(t,s) + \bold{II}(t,s),
}
Since $|\xi_2| \leq m$, it yields
\bes{
    \bold{I}(t,s) \leq &    \int_t^{(T-d) \wedge (t+s+d)} \big|\xi_2(x,0)\big| \big| \xi_3(x,t+s+\eta-x,0)
     - \xi_3(x,t+s-x,0) \big|dx \\
     \leq & m  \int_t^{(T-d) \wedge (t+s+d)} \epsilon(x) dx,
}
where $\epsilon$ is defined as
\bes{
\label{eq:eps_definition}
    \epsilon(x) =&\underset{\substack{s,r\\ \in [-d,0]^2}}{\sup}\left| \xi_3(x,s,r)- \xi_3(x,s +  \eta,r) \right|+ \underset{s\in [-d,0]}{\sup}\left| \xi_2(x,s)-\xi_2(x,s +  \eta) \right|.
}
Furthermore, as $|\xi_2| \wedge |\xi_3| \leq m$ on $\Dcal_\tau$, we have 
\bes{
    \bold{II}(t,s) \leq &  \int_{(T-d) \wedge (t+s+d) }^{(T-d) \wedge (t+s+\eta+d)} |\xi_2(x,0)\xi_3(x,t+s+\eta-x,0)|dx \\
     \leq & m^2 \eta.
}
Consequently, for any $t \in [T-d-\tau, T-d]$, we obtain
\bes{
    \label{eq:xi_2_ineq}
    \underset{s}{\sup}\left|\xi_2(t,s)- \xi_2(t,s+\eta) \right| \leq m^2 \eta + m \int_t^{T-d} \epsilon(r)dr.
}
Looking at the equation of $\xi_3$ in system \eqref{eq:system_integral_form_first_slice}, we obtain in a similar manner 
\bes{
    \label{eq:xi_3}
    \left|\xi_3(t,s,r)-\xi_3(t,s+\eta,r) \right| \leq& |b| \bold{I}(t,s,r) + \sigma^{-2}\bold{II}(t,s,r). \\
}
An application to the triangle inequality combined with \eqref{eq:xi_2_ineq} and the Lipschitzianity of $\xi_1$ leads to 
\bes{
    \label{eq:I}
    \bold{I}(t,s,r) \leq&  | \xi_2((T-d) \wedge (t+(s+\eta) \wedge r+d),(s+\eta-r) \vee (r-(s+\eta))-d) \\
    &- \xi_2((T-d) \wedge (t+s \wedge r+d),(s-r) \vee (r-s)-d)|\\
     \leq & (\kappa + m^2(1+ \sigma^{-2})) \eta + m \int_{(T-d) \wedge (t+s \wedge r+d)}^{T-d} \epsilon(x)dx\\
     \leq & (1+2\kappa) \eta + m \int_{t}^{T-d} \epsilon(x)dx.
} 
Furthermore
\bes{
    \label{eq:II}
    \bold{II}(t,s,r) \leq& \Big|\int_{t}^{(T-d) \wedge (t+(s+\eta) \wedge r+d)} \xi_3(x,t+s+\eta-x,0)\xi_3(x,0,t+r-x)dx 
    \\ 
    &-\int_{t}^{(T-d) \wedge (t+s \wedge r+d)} \xi_3(x,t+s-x,0)\xi_3(x,0,t+u-x)dx \Big| 
    \\
    \leq & \int_{t}^{(T-d) \wedge (t+s \wedge r+d)} |\xi_3(x,0,t+r-x)| |\xi_3(x,t+s-x,0) \\
    &- \xi_3(x,t+(s+\eta)-x,0)| dx 
    \\ 
    & +  \int^{(T-d) \wedge (t+(s+\eta) \wedge r+d)}_{(T-d) \wedge (t+s \wedge r+d)} | \xi_3(x,t+(s+\eta)-x,0)\xi_3(r,0,t+r-x) | dx
    \\
     \leq & m^2 \eta + \int_t^{T-d} \epsilon(r) dr 
}
Thus, inequality \eqref{eq:I} together with \eqref{eq:II} and \eqref{eq:xi_3} yield the existence of a positive constant $c>0$, independent of $\eta$, such that
\bes{
     \sup_{\substack{s,r \\ \in [-d,0]^2}} \left|\xi_3(t,s,r)-\xi_3(t,s+\eta,r) \right| \leq & c \left( \eta +  \int_t^{T-d}  \epsilon(r) dr \right),
}
which, combined with \eqref{eq:xi_2_ineq} leads, for any $t \in [T-d-\tau, T-d]$, to 
\bes{
    \epsilon(t) \leq c \left(\eta + \int_t^{T-d} \epsilon(r)dr \right).
}
Consequently, an application to Gronwall's lemma yields $\epsilon(t) \leq m' \eta$ on $[T-d-\tau, T-d]$, with $m'>0$.  Thus, $\xi_2$ and $\xi_3$ are Lipschitz in the s-variable. The arguments for showing that $\xi_2$ and $\xi_3$ are Lipschitz in the $t$-variable and $\xi_3$ Lipschitz in the $r$-variable follow the same line.
\end{proof}

\begin{lemma}
\label{L:slice_2}
There exists a unique absolutely continuous solution $\xi=(\xi_1, \xi_2, \xi_3)$ of \eqref{eq:system_integral_form_first_slice} on $[T-2 d, T-d]\times [-d,0]^2$ such that $\xi_1\geq 1- d \left( \frac{b}{\sigma} \right)^2 >0$.
\end{lemma}
\begin{proof}
  Let $\theta \in [T - 2d, T - d)$ denote the lower limit of all $\tau$'s such that there exists an absolutely continuous solution $(\xi_1, \xi_2, \xi_3)$ to \eqref{eq:system_integral_form_first_slice} on $[\theta, T-d]$. Assume $\theta > T - 2 d$. From Lemma \ref{lemma:lipschitz_first_slice}, $\xi_1$, $\xi_2$ and $\xi_3$ are Lipschitz in each variable and thus admit a limit, when $t \to \theta$, which is Lipschitz. Therefore, the argument of Lemma \eqref{lemma:existence_first_slice} can be repeated to extend the existence and uniqueness of the solution of system \eqref{eq:system_integral_form_first_slice} on $[\xi, T-d]$ for $T-2d \leq \xi < \theta $. As a result, we necessarily have $\theta = T - 2 d$. It remains to prove that $0<\xi_1$. For this, note that since $\xi_1$ is solution to \eqref{eq:system_integral_form_first_slice}, we have 
  \bes{
    \label{eq:xi_1_xi_2}
    \|\xi_1 - 1\|_\infty \leq \frac{d}{\sigma^{2}}  \sup_{\substack{t \in \\ [T-2 d, T-d] }} |\xi_2(t,0)|^2.
  }
    By injecting the boundary condition \eqref{eq:Ei_terminal_first_slice} into the system \eqref{eq:system_integral_form_first_slice}, one notes  that $t \in [T-2 d, T-d] \mapsto \xi_2(t, 0)$ is solution to
    \bes{
        \xi_2(t, 0) = b - \sigma^{-2}\int_t^{T-d} \xi_2(x, 0)\xi_3(x,t-x, 0) dx, \qquad T-2 d \leq t \leq  T-d.
    }
    Or, for every $t \in [T-2d , T-d]$,  $f_t : x\in [t, T-d] \mapsto f_t(x) :=  \xi_3(x,t-x,0)$ takes only positive values as $f_t$ is solution to the system
    \bes{
    f_t(x) &= b^2 - \sigma^{-2}\int_x^{T-d} f_t(u) \xi_3(u, 0,x-u) du, \qquad x \in [t,T-d],\\
    f_t(T-d) &= b^2,
    }
    which can be proven to admit, through a contraction proof in the Banach space $C([t,T-d], \R)$, a unique positive solution since $\xi$ and its derivatives are bounded. Similarly, we also have $\xi_2(t,0) \geq 0$ for any $t \in [T-2d, T-d]$. As a result, we have $\text{\textit{sign}}(\xi_2) = \text{\textit{sign}}(b)$ and 
    \bes{
        \label{eq:xi_2_b}
         \sup_{\substack{t \in \\ [T-2 d, T-d] }} |\xi_2(t,0)| \leq |b|.
    }
 Consequently, \eqref{eq:xi_1_xi_2} and \eqref{eq:xi_2_b} yield that for any $T-2 d \leq t \leq  T-d$, we have $\xi_1 \geq 1- d \left( \frac{b}{\sigma} \right)^2 = a_2 >0$ as $\Ncal(d, b, \sigma)$ is assumed to be greater than $2$.  
\end{proof}
\vspace{4mm}
Finally, by setting $P_{11}(t) = \xi_1(t)$, $P_{12}(t,s) = \xi_2(t,s)$, $P_{22}(t,s,r)=\xi_3(t,s,r)$  and $P_{\hat{22}}(t,s) = \xi_1(t+s+d)$ for any $(t,s,r) \in [T-2 d, T-d] \times [-d, 0]^2$, Lemma \ref{L:slice_2} yields the existence and uniqueness of a solution $P$ to \eqref{eq:a}-\eqref{eq:b}-\eqref{eq:c} in the sense of definition \ref{def:sol_E_i_simplified} on $[T-2 d, T-d]$. The concatenation of the unique solution of \eqref{eq:a}-\eqref{eq:b}-\eqref{eq:c}  on $[T-d, T]$ and $[T-2d, T-d]$ leads to a unique solution on $[T-2d, T]$.

\subsection{From slice $[T- n d, T]$ to $[T-(n+1)d, T]$}
Let $n$ be an integer such that $2 \leq n < \Ncal(d, b, \sigma)$. Assume that there exists a solution $P$ to \eqref{eq:a}-\eqref{eq:b}-\eqref{eq:c} in the sense of Definition \ref{def:sol_E_i_simplified} on $[T-nd, T]$ such that $0<a_n \leq P_{11}(t) \leq 1$, for any $t \geq T-nd$. Recall the Definition \eqref{def:a_n} of $(a_n)_{n \geq 0}$. Consider the following system on $[T-(n+1)d, T-n d]\times [-d, 0]^2$
\bes{
\label{eq:system_n}
P_{11}(t) =& P_{11}(T-nd)- \int_{t}^{T-nd} \frac{P_{12}(x,0)^{2}}{  \sigma^2 P_{11} (x+d)}dx, 
\\
P_{12}(t,s) =& b P_{11}((T-nd) \wedge (t+s+d))-\int_{t}^{(T-n d) \wedge (t+s+d)} \frac{P_{12}(x,0)P_{22}(x,t+s-x,0)}{\sigma^2 P_{11} (x+d)}dx,\\
P_{22}(t,s,r)  =& b P_{12}((T-nd) \wedge (t+s \wedge r+d),(s-r) \vee (r-s)-d)
\\
&-\int_{t}^{(T-nd) \wedge (t+s \wedge r+d)} \frac{P_{22}(x,t+s-x,0)P_{22}(x,0,t+r-x)}{\sigma^2 P_{11} (x+d)}dx.
}
Note that this system is the same as \eqref{eq:system_integral_form_first_slice}, the only difference being the term $x\in [T-(n+1)d, T-nd] \mapsto P_{11}(x+d)$ which comes from the previous slice $[T-nd, T-(n-1)d]$. Therefore, it can be considered as a positive continuous coefficient by induction hypothesis. As result, existence and uniqueness on $[T-(n+1)d, T-nd]$ can be proven in the same fashion as in Lemmas \ref{lemma:existence_first_slice}-\ref{lemma:lipschitz_first_slice}-\ref{L:slice_2}. It remains to prove that $P_{11}(t)\geq a_{n+1}$ for any $t\in [T-(n+1)d, T-nd]$. As in Lemma \ref{L:slice_2}, and by using the induction hypothesis, we have  
\bes{
    \label{eq:bound_E_2}
    |P_{12}(t, -d)| \leq |bP_{11}(T-nd)| \leq |b|, \qquad t \in  [T-(n+1)d, T-nd].
}
Furthermore, $P_{11}$ satisfies \eqref{eq:system_n} on $[T-(n+1)d, T-nd]$,  which, combined with $P_{11} \geq a_n$ on $[T-nd, T-(n-1)d]$ and \eqref{eq:bound_E_2} yields
\bes{
    P_{11}(t) &\geq P_{11}(T-nd) - \frac{d}{a_n} \left( \frac{b}{\sigma}\right)^2 \\
    &\geq a_n - \frac{d}{a_n} \left( \frac{b}{\sigma}\right)^2  = a_{n+1}>0,
}
for any $t \in [T-(n+1) d, T-n d]$, which ends the proof.

\vspace{2cm}

\end{document}